\documentclass[10pt]{amsart}
\usepackage[nohug]{diagrams}
\usepackage{amscd}
\usepackage[mathscr]{eucal}
\usepackage{amsfonts}
\usepackage{amsmath}
\usepackage{amsthm}
\usepackage{amssymb}
\usepackage{latexsym}
\usepackage{tabularx,array}
\usepackage[all]{xy}
\newfont{\msam}{msam10}

\newtheorem{theorem}[]{Theorem}
\newtheorem{proposition}[]{Proposition}
\newtheorem{corollary}[]{Corollary}
\newtheorem{lemma}[]{Lemma}
\theoremstyle{definition}
\theoremstyle{Proposition}
\newtheorem{definition}[]{Definition}
\newtheorem{remark}[]{Remark}

\newtheorem{prop}[theorem]{Proposition}

\let\nc\newcommand

\def\bthm{\begin{theorem}}
\def\ethm{\end{theorem}}
\def\blemma{\begin{lemma}}
\def\elemma{\end{lemma}}
\def\bproof{\begin{proof}}
\def\eproof{\end{proof}}
\def\bprop{\begin{proposition}}
\def\eprop{\end{proposition}}
\def\bcor{\begin{corollary}}
\def\ecor{\end{corollary}}
\nc{\la}{\label}
\def\N{\mathbb{N}}

\def\c{\mathbb{C}}

\def\L{\boldsymbol{L}}

\def\Alg{\mathtt{Alg}}

\def\cAlg{\mathtt{Comm\,Alg}}

\def\Sets{\mathtt{Sets}}
\def\DGA{\mathtt{DGA}}

\def\cDGA{\mathtt{CDGA}}

\def\Ho{{\mathtt{Ho}}}
\newcommand{\bL}{\boldsymbol{\Lambda}}
\nc{\FT}{\mathcal{C}}
\nc{\Ob}{{\rm Ob}}
\nc{\Hom}{{\rm{Hom}}}
\nc{\HOM}{\underline{\rm{Hom}}}
\nc{\DER}{\underline{\rm{Der}}}
\nc{\END}{\underline{\rm{End}}}
\nc{\bSym}{\bold{\Lambda}}
\nc{\Ext}{{\rm{Ext}}}
\nc{\Rep}{{\rm{Rep}}}
\nc{\DRep}{{\rm{DRep}}}
\nc{\NCRep}{\widetilde{\rm{Rep}}}
\nc{\RAct}{{\rm{RAct}}}
\nc{\bs}{\backslash}
\nc{\cn}{{\tt{Cone}}}
\nc{\n}{\natural}
\nc{\B}{\Omega}
\nc{\Ba}{\overline{\mathrm{B}}}
\nc{\bC}{\overline{C}}
\nc{\EXT}{\underline{\rm{Ext}}}
\nc{\TOR}{\underline{\rm{Tor}}}
\def\H{\mathrm H}
\def\HC{\mathrm{HC}}

\def\rHC{\overline{\mathrm{HC}}}
\def\CC{\mathrm{CC}}

\nc{\End}{{\rm{End}}}
\nc{\GL}{{\rm{GL}}}
\nc{\gl}{{\mathfrak{gl}}}
\nc{\PGL}{{\rm{PGL}}}
\nc{\SL}{{\rm{SL}}}
\nc{\PSL}{{\rm{PSL}}}
\nc{\ad}{{\rm{ad}}}
\nc{\Ad}{{\rm{Ad}}}
\nc{\dlim}{\varinjlim}
\nc{\plim}{\varprojlim}

\newcommand{\HH}{{\rm{HH}}}

\newcommand{\Spec}{{\rm{Spec}}}

\newcommand{\id}{{\rm{Id}}}

\newcommand{\Tr}{{\rm{Tr}}}

\newcommand{\into}{\,\hookrightarrow\,}

\newcommand{\onto}{\,\twoheadrightarrow\,}
\newcommand{\sonto}{\,\stackrel{\sim}{\twoheadrightarrow}\,}
\def\der{\mathtt{Der}}

\newcommand{\rar}{\rightarrow}

\def\ldb{\mathopen{\{\!\!\{}}
\def\rdb{\mathclose{\}\!\!\}}}
\def\ldbg{\mathopen{\bigl\{\!\!\bigl\{}}
\def\rdbg{\mathclose{\bigr\}\!\!\bigr\}}}

\newcommand{\CH}{\mathrm{C}}

\allowdisplaybreaks

\begin{document}

\title{Noncommutative Poisson Structures, Derived Representation Schemes and Calabi-Yau Algebras}

\author{Yuri Berest}
\address{Department of Mathematics,
 Cornell University, Ithaca, NY 14853-4201, USA}
\email{berest@math.cornell.edu}
\author{Xiaojun Chen}
\address{School of Mathematics, Sichuan University, Chengdu 610064, China}
\email{xjchen@scu.eu.cn}
\author{Farkhod Eshmatov}
\address{Max Planck Institute for Mathematics, Vivatsgasse 7, 53111 Bonn, Germany}
\email{eshmatov@mpim-bonn.mpg.de}
\author{Ajay Ramadoss}
\address{Departement Mathematik,
Eidgenossische TH Z\"urich,
8092 Z\"urich, Switzerland}
\email{ajay.ramadoss@math.ethz.ch}

\maketitle

\section{Introduction}
Recall that a Poisson structure on a commutative algebra $A$ is a Lie bracket $\,\{\,\mbox{--}\,,\,\mbox{--}\,\}:\,A \times A \to A \,$ satisfying the Leibniz rule $\,\{a,bc\} = b\{a,c\} + \{a,b\}c \,$ for all $\,a,\,b,\,c \in A $. For noncommutative algebras, this definition is known to be too restrictive: if $A$ is a noncommutative domain (more generally, a prime ring), any Poisson bracket on $A$ is a multiple of the commutator $ [a,b] = ab-ba $ (see [FL], Theorem~1.2). Motivated by recent work on noncommutative geometry (see \cite{Ko, Gi, BL, CBEG, vdB}), Crawley-Boevey proposed in \cite{CB} a different notion of the Poisson structure on an algebra $A$ that
agrees with the above definition for commutative algebras and has surprisingly nice categorical properties.
The idea of \cite{CB} was to find the weakest structure on $A$ that induces natural Poisson structures on the moduli spaces of 
finite-dimensional semisimple representations of $A$. It turns out that such a weak Poisson structure is given by a Lie bracket on the 0-th cyclic homology $\, \HC_0(A) = A/[A,A] \,$ satisfying some extra conditions; it is thus called in \cite{CB} an {\it $ H_0$-Poisson structure}. The very terminology of \cite{CB} suggests that there might exist a `higher' homological generalization of this construction. The aim of this paper is to show that this is indeed the case: our main construction yields a graded (super) Lie algebra structure on the full cyclic homology of $A\,$:
\begin{equation*}
\{\,\mbox{--}\,,\,\mbox{--}\,\}:\ \HC_{\bullet}(A) \times \HC_{\bullet}(A) \to \HC_{\bullet}(A)
\end{equation*}
that satisfies certain properties and restricts to Crawley-Boevey's $\H_0$-Poisson structure on $ \HC_0(A) $. 
We call such structures the {\it derived Poisson structures} on $A$.

To explain our results in more detail we first recall the main theorem of \cite{CB}.
Let $A$ be an associative unital algebra and let $ V $ be a finite-dimensional vector space, both
defined over a field $k$ of characteristic zero. The classical representation scheme parametrizing the
$k$-linear representations of $A$ in $V$ can be defined as the functor on the category of commutative algebras
\begin{equation}
\la{rep}
\Rep_V(A):\ \cAlg_{k} \to \Sets\ ,\quad B \mapsto \Hom_{\Alg_k}(A,\, B \otimes_k \End\,V)\ .
\end{equation}
It is well known that \eqref{rep} is representable, and we denote the corresponding commutative algebra by
$\, k[\Rep_V(A)] \,$. The group $ \GL(V) $ acts naturally on the scheme $ \Rep_V(A) $, with orbits
corresponding to the isomorphism classes of representations.
The closed orbits correspond to the classes of semisimple representations and are parametrized
by the affine quotient scheme $\,\Rep_V(A)/\!/\GL(V) = \Spec\,k[\Rep_V(A)]^{\GL(V)} $  (see, e.g., \cite{K}).
Now, there is a natural trace map
\begin{equation}
\la{trr}
\Tr_V:\ \HC_0(A) \to k[\Rep_V(A)]^{\GL(V)}
\end{equation}
defined by taking the characters of representations. In terms of \eqref{trr}, we can state
the main result of \cite{CB} as follows.
\begin{theorem}[\cite{CB}, Theorem~1.6]
\la{CBT}
Given an $ H_0$-Poisson structure on $A$, for each $V $, there exists a unique Poisson structure on
$ \Rep_V(A)/\!/\GL(V) $ satisfying
$$
\{\Tr_V(a),\,\Tr_V(b)\} = \Tr_V(\{a,\,b\})\ ,\quad \forall\ a, b \in \HC_0(A)\ .
$$
\end{theorem}

Our generalization of Theorem~\ref{CBT} is based on results of the recent paper \cite{BKR},
where the character map \eqref{trr} is extended to higher cyclic homology. We briefly
review these results referring the reader to \cite{BKR} (and Section~\ref{S2} below) for details.
Varying $A$ (while keeping $V$ fixed) one can regard the representation functor \eqref{rep}
as a functor on the category $\Alg_k$ of algebras. This functor can then extended to the
category $ \DGA_k $ of differential graded (DG) algebras, and the scheme $ \Rep_V(A) $ can be
derived by replacing $A$ by its almost free DG resolution in $ \DGA_k $. The fact that the result is
independent of the choice of resolution was first proved in \cite{CK}. In~\cite{BKR}, we gave
a more conceptual proof, using Quillen's theory of model categories \cite{Q1}, and found a simple
algebraic construction for the total derived functor of $ \Rep_V $. When applied to
$ A $, this derived functor is represented (in the homotopy category of DG algebras) by a commutative
DG algebra $ \DRep_V(A) $. The homology of  $ \DRep_V(A) $ depends only on $\,A\,$ and $\,V\,$,
with $ \H_0[\DRep_V(A)] $ being isomorphic to $ k[\Rep_V(A)] $. Following \cite{BKR}, we will
refer to $ \H_\bullet[\DRep_V(A)] $ as the {\it representation homology} of $A$ and denote it by
$ \H_\bullet(A, V) $. The action of $ \GL(V)$ on $ \Rep_V(A) $ naturally extends to $ \DRep_V(A) $,
and there is an isomorphism $\, \H_\bullet[\DRep_V(A)^{\GL(V)}] \cong \H_\bullet(A, V)^{\GL(V)} $.
Now, one of the key results of \cite{BKR} is a construction of the canonical trace
maps
\begin{equation}
\la{trr1}
(\Tr_V)_n:\, \HC_n(A) \to \H_n(A,V)^{\GL(V)}\ ,\quad \forall\,n\ge 0\ ,
\end{equation}
extending \eqref{trr} to the higher cyclic homology\footnote{We will review this
construction in Section~\ref{S2} below.}. In terms of \eqref{trr1},
the main result of the present paper can be stated as a direct generalization of Theorem~\ref{CBT}.
\begin{theorem}
\la{t3s2int}
Given a derived Poisson structure on $A$, for each $ V $, there is a unique graded Poisson
bracket on the graded commutative algebra $ \H_\bullet(A, V)^{\GL(V)} $ such that
$$
\{(\Tr_V)_\bullet(\alpha),\,(\Tr_V)_\bullet(\beta)\} =
 (\Tr_V)_{\bullet}(\{\alpha,\,\beta\})\ ,\quad \forall\ \alpha, \beta \in \HC_\bullet(A)\ .
$$
\end{theorem}
In fact, we will prove a more refined result (Theorem~\ref{NCPoiss}), of which Theorem~\ref{t3s2int}
is an easy consequence. Our key observation is that, when extended properly to the category of DG
algebras, the weak Poisson structures behave well with respect to homotopy (in the sense that the
homotopy equivalent Poisson structures on $A$ induce, via the derived representation functor, homotopy equivalent
DG Poisson algebra structures on $\DRep_V(A) $). Working in the homotopy-theoretic framework allows us to give a
precise meaning to the claim that the derived Poisson structures are indeed the {\it weakest} structures on $A$
inducing the usual (graded) Poisson structures under the representation functor (see Remark~\ref{R11}).

The paper is organized as follows. In Section~\ref{S2}, we review basic definitions and results of
\cite{BKR} and \cite{BR} needed for the present paper. In Section~\ref{G}, we extend Crawley-Boevey's
definition of a NC Poisson structure to the category of DG algebras and introduce a relevant notion of homotopy
for such structures. We also prove our first main result (Theorem~\ref{NCPoiss}) in this section.
In Section~\ref{G.2}, we then propose the definition of a noncommutative $P_{\infty}$-algebra
extending the results of Section~\ref{G} to strong homotopy algebras. We show that a noncommutative $P_{\infty}$-algebra structure on $A$ induces a $P_{\infty}$-structure on $\DRep_V(A)$ and that the homotopy equivalent noncommutative $P_{\infty}$-algebra structures induce homotopy equivalent $P_{\infty}$-structures on $\DRep_V(A)$. This result is part of
Theorem~\ref{NCGinfRep}, which is the second main result of this paper. The proof of Theorem~\ref{NCGinfRep} is parallel
to the proof of Theorem~\ref{NCPoiss}, however the calculations are technically more complicated.
Finally, Section~\ref{s3} provides an interesting class of examples of derived Poisson structures. These examples arise from $n$-cyclic coalgebras (through Van den Bergh's double bracket construction) and include, in particular,
linear duals of finite-dimensional $n$-cyclic algebras. The main result of Section~\ref{s3} -- Theorem~\ref{t1s3} --
shows that there is a natural double Poisson algebra structure on the cobar construction of any cyclic
coassociative DG coalgebra. The finite-dimensional $n$-cyclic algebras are known to be a special case of $n$-Calabi-Yau categories in the sense of~\cite{KS, Cos}; our results imply that these algebras carry noncommutative $(2-n)$-Poisson structures.
We conclude with a few remarks on string topology; these remarks clarify the relation of the present paper to the recent work of two of the current authors (X.~Ch. and F.~E.) with W.~L.~Gan (see \cite{CEG}).

\subsection*{Acknowledgements}{\footnotesize
We thank Travis Schedler for useful correspondence, in particular
for explaining to us the proof of Proposition~\ref{pG.2.2.2}.
The results of this paper were announced at the conference on Mathematical Aspects of Quantization held at Notre Dame University
in June 2011. The first author (Yu.~B.) would like to thank the organizers, in particular Michael Gekhtman and Sam Evens, for inviting him to this conference and giving an opportunity to speak. 
The work of Yu.~B. was partially supported by NSF grant DMS 09-01570;
the work of A.~R. was supported by the Swiss National Science Foundation (Ambizione Beitrag Nr. PZ00P2-127427/1).}

\section{The Derived Representation Functor and Higher Trace Maps}
\la{S2}
In this section, we review basic definitions and results of \cite{BKR} and \cite{BR}. Our purpose is to give a short and 
readable survey which goes slightly beyond the preliminaries for the present paper.
This survey evolved from notes of the talk given by the first author at the Notre Dame conference on quantization.

\subsection{Representation functors}
Throughout, $k$ denotes a base field of characteristic zero.
Let $\DGA_k$ be the category of associative DG algebras over $k$ equipped with differential
of degree $-1$, and let $\cDGA_k$ be its full subcategory consisting of commutative DG algebras.
The inclusion functor $\, \cDGA_k \into \DGA_k \,$ has an obvious left adjoint which assigns 
to a DG algebra $ A $ its abelianization; we denote it by
\begin{equation}\la{ab}
(\,\mbox{--}\,)_{\n\n} :\ \DGA_k \rar \cDGA_k \ ,\quad A \mapsto A/ \langle [A,A]\rangle \ .
\end{equation}
Now, given a finite-dimensional $k$-vector space $ V $, we introduce the following functor
\begin{equation}\la{root}
\sqrt[V]{\,\mbox{--}\,}\,:\,\DGA_k \rar \DGA_k\ ,\quad
A \mapsto  (A \ast_k \End\, V)^{\End\, V}\ .
\end{equation}
Here $\,A \ast_k \End\,V \,$ denotes the free product of $A$ with the endomorphism algebra of $V$
and $\,(\,\mbox{--}\,)^{\End\,V} $ stands for the centralizer of $\End\,V $ as the subalgebra
in that free product. Combining \eqref{ab} and \eqref{root}, we define
\begin{equation}
\la{rootab}
(\,\mbox{--}\,)_V :\ \DGA_k \rar \cDGA_k\ ,\quad A \mapsto A_V := (\!\sqrt[V]{A})_{\n\n}\ .
\end{equation}

\begin{theorem}[\cite{BKR}, Theorem~2.2]
\la{t1s1}
For any $\,A \in \DGA_k\,$, the DG algebra $\,A_V\,$ represents the functor
$$
\Rep_V(A):\ \cDGA_k \to \Sets\ ,\quad B \mapsto \Hom_{\DGA_k}(A,\, B \otimes_k \End\,V)\ .
$$
\end{theorem}

\noindent
Theorem~\ref{t1s1} implies that there is a bijection
\begin{equation}
\la{S2E12}
\Hom_{\cDGA_k}(A_V,\,B) = \Hom_{\DGA_k}(A,\, B \otimes_k \End\,V)\ ,
\end{equation}
functorial in $ A \in \DGA_k $ and $B \in \cDGA_k $. Informally, it suggests that
$ A_V = k[\Rep_V(A)] $ should be thought of as a DG algebra of functions on the affine DG scheme parametrizing
the representations of $A$ in $V$. Letting $\,B = A_V\,$ in \eqref{S2E12}, we get a canonical DG algebra
homomorphism
\begin{equation}
\la{e1s1}
\pi_V\,:\, A \to A_V \otimes \End\,V \ ,
\end{equation}
which is the universal representation of $A$. Furthermore, for  $ g \in \GL_k(V) $, we have
a unique automorphism of $\,A_V\,$ corresponding under the adjunction \eqref{S2E12} to the composite map
$$
A \xrightarrow{\pi_V} A_V \otimes \End\,V \xrightarrow{\text{Ad}(g) \otimes \id} A_V \otimes \End\,V  \ .
$$
This defines an action of $\GL_k(V)$ on $A_V$ by DG algebra automorphisms,
that is functorial in $A$. Thus, we can introduce the $\GL_k(V)$-invariant subfunctor of \eqref{rootab}:
\begin{equation}
\la{vgl}
(\,\mbox{--}\,)_V^{\GL} :\ \DGA_k \to \cDGA_k\ ,\quad
A \mapsto A_V^{\GL_k(V)}\ .
\end{equation}

\subsubsection{} The categories $\DGA_k$ and $\cDGA_k$ carry natural closed model structures (in the sense of
Quillen \cite{Q1}). The weak equivalences in these model categories are the quasi-isomorhisms and the
fibrations are the degreewise surjective maps. The cofibrations are characterized in abstract terms:
as the morphisms satisfying the left lifting property with respect to the acyclic fibrations (see, e.g., \cite{H}).
Every DG algebra $\,A \in \DGA_k \,$ has a cofibrant
resolution which is given by a surjective quasi-isomorphism $\, QA \sonto A \,$, with $ QA $ being a cofibrant object in $ \DGA_k $.
In particular, if $A $ is concentrated in non-negative degrees (for example, an ordinary algebra $ A \in \Alg_k $), any almost
free resolution $\,R \sonto A \,$ is cofibrant in $\DGA_k $.
Replacing DG algebras by their cofibrant resolutions one defines the homotopy category
$ \Ho(\DGA_k) $, in which the morphisms are given by the homotopy classes of morphisms between
cofibrant objects in $ \DGA_k $. The category $ \Ho(\DGA_k) $ is equivalent to the (abstract) localization of
the category $ \DGA_k $ at the class of weak equivalences. The corresponding localization functor
$\,\DGA_k \to \Ho(\DGA_k)\,$ acts as the identity on objects while mapping each morphism
$ f: A \to B $ in $ \DGA_k $ to the homotopy class of its cofibrant lifting 
$ Qf: QA \to QB $ (see, e.g., \cite{DS}).

We can now state one of the main results of \cite{BKR} which combines (part of) Theorem~2.2 and
Theorem~2.6 of~{\it loc. cit.}
\begin{theorem}[\cite{BKR}]
\la{t2s1}
$(a)$ The functor \eqref{rootab} has a total left derived functor
$$
\L(\,\mbox{--}\,)_V :\ \Ho(\DGA_k) \to \Ho(\cDGA_k)\ , \quad A \mapsto (QA)_V \,,\ f \mapsto (Qf)_V \ ,
$$
which is adjoint to the derived functor $\,\End\,V\,\otimes\,\mbox{--}\,:\, \Ho(\cDGA_k) \to \Ho(\DGA_k)\,$.

$(b)$  The functor \eqref{vgl} has a total left derived functor
$$
\L(\,\mbox{--}\,)_V^{\GL} :\ \Ho(\DGA_k) \to \Ho(\cDGA_k)\ , \quad A \mapsto (QA)_V^{\GL} \,,\ f \mapsto (Qf)_V^{\GL} \ .
$$
Here $QA$ is any cofibrant replacement of $A$ and $Qf$ is the corresponding cofibrant lifting of $f$.
\end{theorem}
The point of Theorem~\ref{t2s1} is that the DG algebras $ (QA)_V $ and  $ (QA)_V^{\GL} $
depend only on $A$ and $V$, provided we view them as objects in the homotopy category $ \Ho(\cDGA_k)$.
In particular, for $ A \in \Alg_k$, we set $ \DRep_V(A)\,:=\, \L(A)_V \,$ and
define\footnote{Sometimes, we will abuse this notation letting $\H_{\bullet}(A,V)$
denote $ \H_{\bullet}[\L(A)_V]$ for any DG algebra $ A \in \DGA_k$.}
$ \H_{\bullet}(A,V)\,:=\, \H_{\bullet}[\DRep_V(A)] $. This last object is a graded commutative
algebra which we call the {\it representation homology} of $A$.
Using the standard adjunction \eqref{S2E12}, it is not difficult to show
that $\,\H_0[\DRep_V(A)] \cong k[\Rep_V(A)]\,$ whenever $A$ is an ordinary algebra (see \cite{BKR}, 2.3.4).
In addition, we have the following property which shows that
homology commutes with taking invariants.
\begin{prop}[\cite{BKR}, Theorem 2.6]
\la{p2s1}
For any $ A\in \DGA_k $, there is a natural isomorphism
of graded commutative algebras
$$
\H_{\bullet}[\L(A)_V^{\GL}] \,\cong\, \H_{\bullet}(A,V)^{\GL_k(V)} \ .
$$
\end{prop}

\subsection{Higher traces} We now  construct the trace maps \eqref{trr1} relating
cyclic homology to representation homology. Given an associative DG algebra $ R $ with the identity
element $ 1_R \in R $ , we write
$$
\quad R_\n := R/[R,R]\quad ,\quad \FT(R) := R/(k \cdot 1_R + [R,R])\ .
$$
Both $ R_\n $ and $ \FT(R) $ are complexes of vector spaces with differentials induced from $R$. 
If $ A \in \Alg_k $ is an ordinary  algebra, we let $ \HC_\bullet(A) $ and $ \rHC_\bullet(A) $ denote 
its cyclic and reduced cyclic homology, respectively. The precise relation between the two 
is explained in \cite{L}, Sect.~2.2.13; here, we only recall a canonical map
\begin{equation}
\la{crc}
\HC_\bullet(A) \to \rHC_\bullet(A)
\end{equation}
which is induced by the projection of complexes $ \CC_\bullet(A) \onto \CC_\bullet(A)/\CC_\bullet(k) $, where $ \CC_\bullet(A) $ is the Connes cyclic complex computing $ \HC_\bullet(A) $.

\vspace{2ex}

The starting point for our construction is the following well-known result due to Feigin and Tsygan.
\bthm[\cite{FT}, Theorem~1] \la{ftt}
For any $ A \in \Alg_k $, there is an isomorphism of graded vector spaces
$$
\rHC_\bullet(A) \cong \H_\bullet[\FT(R)]\ ,
$$
where $ R = QA $ is a(ny) cofibrant resolution of $A$ in $\DGA_k $.
\ethm
\noindent
For a simple conceptual proof of this theorem, we refer to \cite{BKR}, Section~3.

\vspace{2ex}

Now, for any $ R \in \DGA_k $, consider the composite map
$$
R \xrightarrow{\pi_V} R_V \otimes \End\,V \xrightarrow{\id \otimes \Tr} R_V
$$
where $ \pi_V $ is the universal representation of $R$ in $V$ and $\,\Tr:\,\End\,V \to k \,$ is the
usual matrix trace. It is clear that this map factors through $ R_\n $ and its image lies
in $ R^{\GL}_V $. Hence, we get a map of complexes
\begin{equation}
\la{e2s1}
\Tr_V(R)_{\bullet}:\ R_{\natural} \to R_V^{\GL}\ ,
\end{equation}
which extends by multiplicativity to the map of graded commutative algebras
\begin{equation}
\la{e2s11}
\Tr_V(R)_{\bullet}:\ \bSym(R_{\natural}) \to R_V^{\GL}\ ,
\end{equation}
where $ \bSym $ denotes the graded symmetric algebra over $k$.
We will need the following result which is a generalization of a well-known
theorem of Procesi \cite{P} to the case of DG algebras.
\begin{theorem}[\cite{BR}, Theorem~3.1]
\la{t3s1}
For any $ R \in \DGA_k$, the algebra map \eqref{e2s11} is degreewise surjective.
 \end{theorem}

Now, let $R = QA $ be a cofibrant resolution of an ordinary algebra $A \in \Alg_k $. Then, with identifications
of Proposition~\ref{p2s1} and Theorem~\ref{ftt} and in combination with \eqref{crc}, the map \eqref{e2s1} induces
\begin{equation}
\la{e4s1}
\Tr_V(A)_{\bullet} :\ \HC_{\bullet}(A) \to \H_{\bullet}(A,V)^{\GL_k(V)} \ ,
\end{equation}
which is the higher trace map \eqref{trr1} discussed in the Introduction\footnote{To simplify the notation, we will often
write $ \Tr_V(A)_{\bullet} $ as $ (\Tr_V)_{\bullet} $.}.

\subsection{Stabilization}
We now explain how to `stabilize' the family of maps \eqref{e4s1} passing to the infinite-dimensional
limit $\, \dim_k V \to \infty $.
We will work with unital DG algebras $A$ which are {\it augmented} over $k$. We recall that the
category of such DG algebras is naturally equivalent to the category of non-unital DG algebras,
with $ A $ corresponding to its augmentation ideal $ \bar{A} $. We identify these two categories
and denote them by $ \DGA_{k/k} $. Further, to simplify the notation we take $ V = k^d $ and
identify $\, \End\,V = M_d(k) \,$, $\, \GL(V) = \GL_k(d) \,$; in addition, for $ V = k^d $,
we will write  $ A_V $ as $ A_d $.
Bordering a matrix in $ M_d(k) $ by $0$'s on the right and on the bottom gives an embedding
$\,M_d(k) \into M_{d+1}(k) \,$ of non-unital algebras. As a result, for each $ B \in \cDGA_k $,
we get a map of sets
\begin{equation}
\la{isoun0}
\Hom_{\DGA_{k/k}}(\bar{A},\,M_d(B)) \to \Hom_{\DGA_{k/k}}(\bar{A},\,M_{d+1}(B))
\end{equation}
defining a natural transformation of functors from $ \cDGA_k $ to $ \Sets $.
Since $B$'s are unital and $ A $ is augmented, the restriction maps
\begin{equation}
\la{isoun1}
\Hom_{\DGA_k}(A,\,M_d(B)) \stackrel{\sim}{\to} \Hom_{\DGA_{k/k}}(\bar{A},\,M_d(B)) \ ,\quad
\varphi \mapsto \varphi|_{\bar{A}}
\end{equation}
are isomorphisms for all $ d \in \N $. Combining \eqref{isoun0} and \eqref{isoun1},
we thus have natural transformations
\begin{equation}
\la{isoun}
\Hom_{\DGA_k}(A,\,M_d(\,\mbox{--}\,)) \to \Hom_{\DGA_k}(A,\,M_{d+1}(\,\mbox{--}\,))\ .
\end{equation}
By standard adjunction \eqref{S2E12}, \eqref{isoun} yield an inverse system of
morphisms $\,\{ \mu_{d+1, d}: A_{d+1} \to A_d \} \,$ in $ \cDGA_k $. Taking the limit of this
system, we define
$$
A_{{\infty}} := \varprojlim_{d\,\in\,\mathbb N} A_d \ .
$$
Next, we recall that the group $ \GL(d) $ acts naturally on $ A_d $,
and it is easy to check that $\,\mu_{d+1, d}: A_{d+1} \to A_d\,$ maps the subalgebra
$ A_{d+1}^{\GL} $ of $\GL $-invariants in $ A_{d+1} $ to the subalgebra $ A_d^{\GL}$ of
$\GL $-invariants in $A_d$. Defining $ \GL(\infty) := \varinjlim\, \GL(d) $ through
the standard inclusions $ \GL(d) \into \GL(d+1) $, we extend the actions of $ \GL(d) $
on $ A_d $ to an action of $ \GL(\infty) $ on $ A_{\infty} $ and let $ A^{\GL(\infty)}_{{\infty}} $
denote the corresponding invariant subalgebra. Then one can prove (see~\cite{T-TT})
\begin{equation}
\la{isolim}
A^{\GL(\infty)}_{{\infty}} \cong \varprojlim_{d\,\in\,\mathbb N} A^{\GL(d)}_d \ .
\end{equation}
This isomorphism  allows us to equip $ A^{\GL(\infty)}_{{\infty}} $ with a natural topology:
namely, we put first the discrete topology on each $ A^{\GL(d)}_d $ and equip
$\,\prod_{d \in \N} A^{\GL(d)}_d \,$ with the product topology; then, identifying
$ A^{\GL(\infty)}_{{\infty}} $ with a subspace in
$\,\prod_{d \in \N} A^{\GL(d)}_d \,$ via \eqref{isolim}, we put on
$ A^{\GL(\infty)}_{{\infty}} $ the induced topology. The
corresponding topological DG algebra will be denoted $ A^{\GL}_{{\infty}} $.

Now, for each $ d \in \N $, we have the commutative diagram
\[
\begin{diagram}[small, tight]
  &       &     \FT(A)     &        & \\
  &\ldTo^{\Tr_{d+1}(A)_\bullet}  &           &  \rdTo^{\Tr_{d}(A)_\bullet} &  \\
A_{d+1}^{\GL} &       & \rTo^{\mu_{d+1, d}}  &        &  A_d^{\GL}
\end{diagram}
\]
where $\, \FT(A) \,$ is the cyclic functor of Feigin and Tsygan ({\it cf.} Theorem~\ref{ftt})
restricted to $ \DGA_{k/k} $.
Hence, by the universal property of inverse limits, there is a morphism of complexes $\, \Tr_{\infty}(A)_\bullet :\, \FT(A) \to A^{\GL}_{{\infty}}\,$ that factors $ \Tr_{d}(A)_\bullet $ for each $ d \in \N $. We extend this morphism to a
homomorphism of commutative DG algebras:
\begin{equation}
\la{trinf}
\Tr_{\infty}(A)_\bullet :\ \bL[\FT(A)] \to A^{\GL}_{{\infty}}\ .
\end{equation}

The following lemma is one of the key technical results of \cite{BR}; it should be compared to
Theorem~\ref{t3s1} in the finite-dimensional case ($ d = \dim_k V$).
\blemma[\cite{BR}, Lemma~3.1]
\la{dense}
The map \eqref{trinf} is {\rm topologically} surjective: i.e., its image is dense in $ A^{\GL}_{{\infty}} $.
\elemma

\vspace{1ex}

\noindent
Letting $ A_\infty^{\Tr}  $ denote the image of \eqref{trinf}, we
define the functor
\begin{equation}
\la{trfun}
(\,\mbox{--}\,)^{\Tr}_{\infty}\,:\ \DGA_{k/k} \to \cDGA_k\ ,\quad A \mapsto A_\infty^{\Tr}\ .
\end{equation}
The algebra maps \eqref{trinf} then give a morphism of functors
\begin{equation}
\la{morfun}
\Tr_{\infty}(\,\mbox{--}\,)_\bullet :\ \bL[\FT(\,\mbox{--}\,)] \to (\,\mbox{--}\,)^{\Tr}_{\infty}\ .
\end{equation}

Now, to state the main result of \cite{BR} we recall that
the category of augmented DG algebras $ \DGA_{k/k} $ has a natural model structure induced from
$ \DGA_k $. We also recall the derived Feigin-Tsygan functor
$\, \L\FT(\,\mbox{--}\,):\, \Ho(\DGA_{k/k}) \to \Ho(\cDGA_k)\, $ inducing the isomorphism of
Theorem~\ref{ftt}.
\bthm[\cite{BR}, Theorem~4.2]
\la{eqfun}
$(a)$ The functor \eqref{trfun} has a total left derived functor
$\,\L(\,\mbox{--}\,)^{\Tr}_{\infty}\,:\ \Ho(\DGA_{k/k}) \to \Ho(\cDGA_k)\,$.

$(b)$ The morphism \eqref{morfun} induces an isomorphism of functors
$$
\Tr_{\infty}(\,\mbox{--}\,)_\bullet :\  \bL[\L\FT(\,\mbox{--}\,)] \stackrel{\sim}{\to}
\L (\,\mbox{--}\,)^{\Tr}_{\infty}\ .
$$
\ethm
\noindent
By definition, $\, \L(\,\mbox{--}\,)^{\Tr}_{\infty} $ is given by
$\,
\L(A)^{\Tr}_{\infty} = (QA)^{\Tr}_{\infty}\,$,
where $ QA $ is a cofibrant resolution of $A$ in $ \DGA_{k/k}\,$. For an ordinary augmented $k$-algebra $A \in \Alg_{k/k} $, we set
$$
\DRep_\infty(A)^\Tr :=  (QA)^{\Tr}_{\infty} \ .
$$
By part $(a)$ of Theorem~\ref{eqfun}, $ \DRep_\infty(A)^\Tr $ is well defined. On the other hand, part $(b)$ implies
\begin{corollary}\la{corf1}
For any $ A \in \Alg_{k/k} $, $\, \Tr_{\infty}(A)_\bullet $ induces an isomorphism of graded commutative algebras
\begin{equation}
\la{funhc}
\bL[\rHC(A)] \cong \H_\bullet[\DRep_\infty(A)^\Tr]\ .
\end{equation}
\end{corollary}

In fact, one can show that $ \H_\bullet[\DRep_\infty(A)^\Tr] $
has a natural structure of a graded Hopf algebra,
and the isomorphism of Corollary~\ref{corf1} is actually an
isomorphism of Hopf algebras. This isomorphism is analogous
to the famous Loday-Quillen-Tsygan isomorphism computing the
stable homology of matrix Lie algebras $ \gl_n(A)$ in terms of
cyclic homology (see \cite{LQ,T}).
Heuristically, it implies that the cyclic homology of an augmented
algebra is determined by its representation homology.

\section{NC Poisson structures and DG representation schemes}
\la{G}
In this section, we propose a definition of a NC Poisson structure on an associative
DG algebra $A \,\in\, \DGA_k $. Our definition generalizes the notion of a noncommutative Poisson structure
in the sense of~\cite{CB}. We show that our noncommutative DG Poisson structures induce (via the natural trace maps)
DG Poisson algebra structures on $A_V^{\GL(V)}$ for all $V$. Subsequently, in the next section, we will introduce
an NC $P_{\infty}$-structure, which is a strong homotopy version of the notion of a NC Poisson structure.
We note that our definition of and results relating to NC Poisson algebras and NC $P_{\infty}$-algebras may be mimicked
to give definitions of, and corresponding results for NC $n$-Poisson algebras and NC $n-\text{P}_{\infty}$ algebras for every $n$. At the level of homology, our construction gives a higher extension of Crawley-Boevey's notion of an
$H_0$-Poisson structure on an algebra $A$. Indeed, suppose that $A$ has a cofibrant resolution
$ R \in \DGA_k\,$ which is equipped with a NC $n$-Poisson structure. Then, this last structure on $R$ induces
a graded Lie algebra structure $\{\mbox{--},\mbox{--}\}_{\n}$ on the (shifted) cyclic homology
$\text{HC}_{\bullet}(A)[n]$, and we will refer to $\,\{\mbox{--},\mbox{--}\}_{\n}\,$ as a
{\it derived $n$-Poisson structure} on $A$.

The following result is a direct generalization of Theorem~\ref{t3s2int} stated in the Introduction.
\begin{theorem} \la{t3s2}
Let $A$ be an algebra equipped with a derived $n$-Poisson structure.
Then, there exists a unique graded $n$-Poisson algebra structure
$\{\mbox{--},\mbox{--}\}$ on
$\mathrm{H}_{\bullet}(A,V)^{\GL}$ such that
 $$(\mathrm{Tr}_V)_{\bullet} (\{\alpha,\beta\}_{\n})\,=\, \{(\mathrm{Tr}_V)_{\bullet}(\alpha),(\mathrm{Tr}_V)_{\bullet}(\beta)\}  $$
for all $\alpha,\beta \,\in\, \mathrm{HC}_{\bullet}(A)$.
\end{theorem}

Theorem~\ref{t3s2} is a consequence of the more fundamental Theorem~\ref{NCPoiss} that we will prove
in this section. In fact, we show (see Theorem~\ref{NCPoiss} (i)) that a NC $n$-Poisson
structure on a DGA $R$ induces DG $n$-Poisson structures on $R_V^{\GL}$ (via natural trace maps) in a functorial manner.
In Theorem~\ref{t3s2}, the graded $n$-Poisson structure on $\mathrm{H}_{\bullet}(A,V)^{\GL}$ is precisely the one induced on homology
by the DG $n$-Poisson structure on $R_V^{\GL}$ coming from Theorem~\ref{NCPoiss}(i).

 Further, we give a reasonable definition of the ``homotopy category'' of NC Poisson algebras and
show a stronger statement (Theorem~\ref{NCPoiss} (ii)) at the level of homotopy categories.

\subsection{NC Poisson algebras}
\la{G.1}
Fix $ A\,\in\, \DGA_k $, and let $\DER(A)^\n $ denote the subcomplex of the DG Lie algebra $\DER(A) $
comprising those derivations whose image is contained in $\,[A,A]\,$. It is easy to see that
$ \DER(A)^\n $ is a DG Lie ideal of $ \DER(A) $, so that $\,\DER(A)_{\natural}:= \DER(A)/\DER(A)^\n $
is a DG Lie algebra.

Now, let $ V $ be a representation of $\,\DER(A)_{\natural}\,$, i.~e. a DG Lie algebra
homomorphism $\, \varrho\,:\, \DER(A)_{\natural} \to \END_k V \,$.

\subsubsection{Definitions} \la{G.1.1} By {\it Poisson structure}
on $V$ we will mean a DG Lie algebra structure on $\,V\,$ whose adjoint representation
$ \mbox{\rm ad}:\,V \to \END_k V $ factors through $ \varrho \,$: i.~e., there is a morphism
of DG Lie algebras $\,i :\, V \rar \DER(A)_{\natural}\,$ such that $\,\mbox{\rm ad} = \varrho \circ i \,$.

For any DG algebra $A$, the natural action of $ \DER(A) $ on $A$ induces a Lie algebra action of
$ \DER(A)_\n $ on $ A_\n $. A {\it NC\ Poisson structure} on $A$ is then,
by definition, a Poisson structure on the representation $A_{\natural}$. It is easy to
see that if $A$ is a commutative DG algebra, a NC Poisson structure on $A$ is exactly
the same thing as a Poisson bracket on $A$.

Let $A$ and $B$ be NC Poisson DG algebras, i.e. objects in $\DGA_k$ equipped with NC Poisson structures.
A {\it morphism} $\,f:\, A \rar B $ of NC Poisson DG algebras is then a morphism $ f: A \to  B $ in
$\DGA_k$ such that $f_{\natural}:\, A_{\natural} \rar B_{\natural} $
is a morphism of DG Lie algebras. We can therefore define the category
$\mathtt{NCPoiss}_k $.

Further, note that if $B$ is a NC Poisson DG algebra and $ \Omega $ is the de Rham algebra of the affine line
({\it cf.}~\cite{BKR}, Section B.4), then $\, B \otimes \Omega $ can be given the structure of a NC Poisson DG algebra
via extension of scalars. Indeed, since $ [B \otimes \Omega,\,
B \otimes \Omega] = [B,B] \otimes\Omega $, we have
$\,(B \otimes \Omega)_{\natural} \cong B_{\natural} \otimes \Omega $. The DG
Lie structure on $B_{\natural} \otimes \Omega $ is simply the one obtained by extending the corresponding
structure on $B_{\natural} $. The structure map $ i\,:\,B_{\natural} \otimes \Omega \rar \DER(B \otimes \Omega)_{\natural} $ is simply the composite map
$$
B_{\natural} \otimes \Omega \xrightarrow{i \otimes \Omega}\DER(B \otimes \Omega)_{\Omega,\natural}
\to \DER(B \otimes \Omega)_{\natural}\ ,
$$
where $\, \DER(B \otimes \Omega)_{\Omega,\n}:= \DER_{\Omega}(B\otimes \Omega)/\DER_{\Omega}(B \otimes \Omega)^\n $
and $\DER_{\Omega}(B\otimes \Omega) $ denotes the DG Lie algebra of $\Omega$-linear derivations from
$ B \otimes \Omega$ into itself, with $\DER_{\Omega}(B \otimes \Omega)^\n $ being the Lie ideal of
derivations whose image is contained in $ B_{\natural} \otimes \Omega $.

We can now introduce the notion of P-homotopy along the lines of~\cite{BKR}, Proposition B.2 and Remark B.4.3.
To be precise, we call two morphisms $\,f,g:A \rar B\,$ in $\mathtt{NCPoiss}$ are
{\it P-homotopic} if there is a morphism $ h:A \rar B \otimes \Omega $ such that $ h(0)=f $ and $h(1)=g$.
It is easy to check that P-homotopy is an equivalence relation on $ \Hom_{\mathtt{NCPoiss}}(A,B)\,$ for
any $A$ and $B$ in $ \mathtt{NCPoiss}_k $. Thus, we can define $ \Ho^*(\mathtt{NCPoiss}) $ to be the
category whose objects are the cofibrant (in $\DGA_k$) DG algebras equipped with NC Poisson structures,
with $\, \Hom_{\Ho^*(\mathtt{NCPoiss})}(A,B)\,$ being the space of P-homotopy classes of morphisms
in $\Hom_{\mathtt{NCPoiss}}(A,B)$.

\noindent
{\bf Notation.} In what follows, for a DG algebra $A$ with a NC Poisson structure, the symbol $[\mbox{--},\mbox{--}]$ shall be used to denote the corresponding Lie bracket on $A_{\n}$. The symbol $\{\mbox{--},\mbox{--}\}_{\n}$ shall be used to denote the induced Lie bracket on $\H_{\bullet}(A_{\n})$.

\subsection{The main theorem}
\la{G.1.2}
The following theorem is the first main result of this paper.
\begin{theorem}
\la{NCPoiss}
$(a)$ The functor $\,(\,\mbox{--}\,)_V^{\GL}:\,\DGA_k \rar \cDGA_k\,$ enriches to give the
following commutative diagram
\begin{equation}
\la{di1}
\begin{diagram}
\mathtt{NCPoiss}_k &\rTo^{(\,\mbox{--}\,)_V^{\GL}} & \mathtt{Poiss}_k\\
 \dTo & &\dTo\\
\DGA_k & \rTo^{(\,\mbox{--}\,)_V^{\GL}} & \cDGA_k
\end{diagram}
\end{equation}
where the vertical arrows are the forgetful functors.

$(b)$ The functor $\,(\,\mbox{--}\,)_V^{\GL}:\,\mathtt{NCPoiss}_k \rar \mathtt{Poiss}_k $ descends to a
functor $\,\L^*(\,\mbox{--}\,)_V^{\GL}:\,\Ho^*(\mathtt{NCPoiss}_k) \rar
\Ho(\mathtt{Poiss}_k)$. Further, $\, \L(\,\mbox{--}\,)_V^{\GL}:\,\Ho(\DGA_k) \rar \Ho(\cDGA_k)\,$
enriches to give a commutative diagram
\begin{equation}
\la{di2}
\begin{diagram}
\Ho^*(\mathtt{NCPoiss}_k) &\rTo^{\L^*(\,\mbox{--}\,)_V^{\GL}} & \Ho(\mathtt{Poiss}_k)\\
 \dTo & &\dTo\\
\Ho(\DGA_k) &\rTo^{\L(\,\mbox{--}\,)_V^{\GL}} & \Ho(\cDGA_k)
\end{diagram}
\end{equation}
\end{theorem}

\vspace{1.5ex}

The next section, Section~\ref{G.1.3}  shall have certain preliminaries we require for the proof of Theorem~\ref{NCPoiss}.
Section~\ref{G.1.4} shall contain the proof of Theorem~\ref{NCPoiss}.

\subsubsection{From DG Lie to DG Poisson algebras} \la{G.1.3}

The following proposition is a (minor) generalization of Theorem~3.3 of~\cite{Sched09}.

\begin{prop}\la{pG.1.3.1}
(i)  One has a functor $$\mathtt{Lie} \rar \mathtt{Poiss},\,\,\,\,\, V \mapsto \bSym(V)$$ from the category of DG Lie algebras to the category of
DG Poisson algebras.\\
(ii) Further, for any DG Poisson algebra $A$ and a morphism $f:V \rar A$ of DG Lie algebras, one has a unique morphism
$\tilde{f}: \bSym(V) \rar A$ of DG Poisson algebras such that $f = \tilde{f} \circ \iota$ where $\iota:V \rar \bSym(V)$ is
 the obvious inclusion.

\end{prop}

\begin{proof}
 Clearly, $\bSym(V)$ is a commutative DG algebra. We extend the Lie bracket on $V$ to a Poisson bracket on
$\bSym(V)$ via the rules\footnote{$|u|$ denotes the degree of a homogenous element $u$ in $\bSym(V)$.}
$$[u,v.w] \,=\, [u,v].w +(-1)^{|u||v|}v.[u,w],\,\,\,\,\,\, [u,v]=(-1)^{|u||v|}[v,u] \,\text{.}$$
That this indeed gives a well defined DG Poisson structure on $\bSym(V)$ is a special case of Proposition~\ref{pG.2.3.1}.
Given a morphism $f:V \rar W$ of DG Lie algebras, one gets the morphism $\bSym(f)\,:\,\bSym(V) \rar \bSym(W)$
in $\cDGA_k$. We verify that $F:=\bSym(f)$ is a morphism of DG Poisson algebras as follows.

First, suppose that $F([u,v])=[F(u),F(v)]$ and $F([u,w])=[F(u),F(w)]$. Then,
$$ F([u,v.w])\,= \, F([u,v].w) +(-1)^{|u||v|}F(v.[u,w]) $$ $$\,= \, F([u,v]).F(w) +(-1)^{|F(u)||F(v)|}F(v).F([u,w])$$ $$\,=\, [F(u),F(v)].F(w) +(-1)^{|F(u)||F(v)|}F(v).[F(u),F(w)] $$ $$\,=\, [F(u),F(v).F(w)] \,=\, [F(u),F(vw)]\, \text{.} $$
Hence, by induction, it suffices to verify that $F([x,y])=[F(x),F(y)]$ on $V$. Since $F|_{V}=f$, the latter is
indeed true. This proves (i).

The above computation proves (ii) as well (with $\tilde{f}$ being the unique extension of $f$ to a morphism $ \bSym(V) \rar
A$ in $\cDGA_k$).
\end{proof}

Two morphisms $f,g:V \rar W$ of DG Lie algebras are {\it L-homotopic}
if there exists a morphism $h: V \rar W \otimes \Omega$ of DG Lie algebras\footnote{The DG Lie structure on $W \otimes \Omega$ is obtained from that on $W$ by extension of scalars.} such that $h(0)=f,\, h(1)=g$.
\begin{lemma} \la{pG.1.3.2}
If $f$ is L-homotopic to $g$, $\bSym(f)$ is P-homotopic to $\bSym(g)$.
\end{lemma}

\begin{proof}
Indeed, the natural inclusion $W \otimes \Omega \hookrightarrow \bSym(W) \otimes \Omega$ is a morphism of DG Lie algebras. Hence, its composition with $h$ is a morphism of DG Lie algebras. By Proposition~\ref{pG.1.3.1} (ii), there exists a unique morphism $\tilde{h}:\bSym(V)
\rar \bSym(W) \otimes \Omega$ of DG Poisson algebras extending $h$. Since $\tilde{h}(0)|_{V}=\bSym(f)|_{V}$, $\tilde{h}(0)=\bSym(f)$ by Proposition~\ref{pG.1.3.1} (ii). Similarly, $\tilde{h}(1)=\bSym(g)$. This proves the desired proposition.
\end{proof}

\subsubsection{Proof of Theorem~\ref{NCPoiss}}\la{G.1.4}

It follows from Proposition~\ref{pG.1.3.1} and Lemma~\ref{pG.1.3.2} that if $f,g:A \rar B$ are P-homotopic morphisms of NC Poisson algebras, then,
$\bSym(f_{\natural}),\bSym(g_{\natural})\,:\,\bSym(A_{\natural}) \rar \bSym(B_{\natural})$ are P-homotopic morphisms of DG Poisson algebras. The following proposition shows that the DG Poisson structure of $\bSym(A_{\natural})$ induces one on $A_V^{\GL}$ via $\mathrm{Tr}_V(A)_{\bullet}$.

\blemma \la{pG.1.4.1}
If $(\mathrm{Tr}_V)_{\bullet}(\beta)=0$, then for any $\alpha\,\in\,\bSym(A_{\natural})$, $(\mathrm{Tr}_V)_{\bullet}([\alpha,\beta])=0$.
\elemma

\bproof
Since
$$[\alpha_1\alpha_2,\beta]=\pm \alpha_1.[\alpha_2,\beta] \pm [\alpha_1,\beta].\alpha_2\,,$$
and since $(\mathrm{Tr}_V)_{\bullet}$ is a morphism in $\cDGA_k$, it suffices to prove the desired
proposition for $\alpha\,\in\,A_{\natural}$. In this case, viewing $\alpha$ as an element of
$A_{\natural}[1]$, consider the element $i(\alpha)\,\in\,\der(A)_{\natural}$. Choose any
 $\partial_{\alpha} \,\in\,\der(A)$ whose image in $\der(A)_{\natural}$ is $i(\alpha)$. By Lemma~\ref{pG.2.0}, there is a (graded) derivation $\psi_{\alpha}$ of $A_V^{\GL}$ such that
 $$(\mathrm{Tr}_V)_{\bullet}(\partial_{\alpha}(\beta))=\psi_{\alpha}((\mathrm{Tr}_V)_{\bullet}(\beta)) $$
 for all $\beta \in A_{\natural}$. Hence,
   $$(\mathrm{Tr}_V)_{\bullet}([\alpha,\beta])=\psi_{\alpha}((\mathrm{Tr}_V)_{\bullet}(\beta)) $$ for all $\beta \in A_{\natural}$. Since
$(\mathrm{Tr}_V)_{\bullet}([\alpha,-])$ as well as $\psi_{\alpha}((\mathrm{Tr}_V)_{\bullet}(-))$ are derivations with respect to
 $(\mathrm{Tr}_V)_{\bullet}$, it follows that
 $$(\mathrm{Tr}_V)_{\bullet}([\alpha,\beta])=\psi_{\alpha}((\mathrm{Tr}_V)_{\bullet}(\beta)) $$
 for all $\beta \in \bSym(A_{\natural})$.
The right hand side of the above equation indeed vanishes when $(\mathrm{Tr}_V)_{\bullet}(\beta)=0$.

\eproof

By Lemma~\ref{pG.1.4.1}, the antisymmetric pairing on $A_V^{\GL}$ given by
$$\{f,g\}:=(\mathrm{Tr}_V)_{\bullet}([(\mathrm{Tr}_V)_{\bullet}^{-1}(f),(\mathrm{Tr}_V)_{\bullet}^{-1}(g)])$$
is well defined. That $\{-,-\}$ equips $A_V^{\GL}$ with the structure of a DG Poisson algebra follows from Proposition~\ref{pG.1.3.1} (i) and Theorem~\ref{t3s1}.

 Further, the following argument shows that if $f:A \rar B$ is a morphism of NC Posson algebras, then
$f_V^{\GL}\,:\,A_V^{\GL} \rar B_V^{\GL}$ is a morphism of DG Poisson algebras. Indeed, since $\bSym(f_{\natural})$ is a morphism of DG Poisson algebras,  $(\mathrm{Tr}_V)_{\bullet} \circ \bSym(f_{\natural})\,=\,f_V^{\GL} \circ (\mathrm{Tr}_V)_{\bullet}$ and
$$\{(\mathrm{Tr}_V)_{\bullet}(\alpha),(\mathrm{Tr}_V)_{\bullet}(\beta)\}=(\mathrm{Tr}_V)_{\bullet}([\alpha,\beta])$$
for all $\alpha,\beta \,\in\, \bSym(A_{\natural})$ (and similarly for $B$),
$$\{f_V^{\GL}((\mathrm{Tr}_V)_{\bullet}(\alpha)),f_V^{\GL}((\mathrm{Tr}_V)_{\bullet}(\beta))\}=\{(\mathrm{Tr}_V)_{\bullet}(\bSym(f_{\natural})(\alpha)),(\mathrm{Tr}_V)_{\bullet}(\bSym(f_{\natural})(\beta))\}$$
$$ =(\mathrm{Tr}_V)_{\bullet}([\bSym(f_{\natural})(\alpha),\bSym(f_{\natural})(\beta)])=(\mathrm{Tr}_V)_{\bullet}\circ \bSym(f_{\natural})([\alpha,\beta])$$
$$=f_V^{\GL}(\{(\mathrm{Tr}_V)_{\bullet}(\alpha),(\mathrm{Tr}_V)_{\bullet}(\beta)\}) \text{.}$$
Theorem~\ref{t3s1} then completes the verification that $f_V^{\GL}\,:\,A_V^{\GL} \rar B_V^{\GL}$ is a morphism of DG Poisson algebras. This completes the proof of Theorem~\ref{NCPoiss} (i). Note that the same argument also shows that if $f,g:A \rar B$ are P-homotopic morphisms of
NC Poisson algebras, then $f_V^{\GL},g_V^{\GL}$ are P-homotopic morphisms of DG Poisson algebras: indeed, if $h:A \rar B\otimes \Omega$ is a P-homotopy between $f$ and $g$, then $h_V^{\GL}:A_V^{\GL} \rar B_V^{\GL} \otimes \Omega$ is a morphism of  DG Poisson algebras by (a trivial modification of) the same argument as above.

Therefore, to complete the proof of Theorem~\ref{NCPoiss} (ii), we only need to verify two assertions:\\
 (a) If $f,g:A \rar B$ are P-homotopic morphisms in $\mathtt{NCPoiss}$, then $\gamma(f)=\gamma(g)$ in $\Ho(\DGA_k)$.\\
 (b) If $f,g:C_1 \rar C_2$ are P-homotopic morphisms in $\mathtt{Poiss}$, then $\gamma(f)=\gamma(g)$ in $\Ho(\mathtt{Poiss})$.\\
 We check (a): let $p_A\,:\, QA \stackrel{\sim}{\rar} A$ be a cofibrant resolution of $A$. Then, $\gamma(f)=\gamma(fp_A)$ and $\gamma(g)=\gamma(gp_A)$ in $\Ho(\DGA_k)$. Since $f,g:A \rar B$ are polynomially M-homotopic to each other (see~\cite{BKR} Remark B.4.3), so are $fp_A$ and $gp_A$. By~\cite{BKR}, Proposition B.2, $\gamma(fp_A)=\gamma(gp_A)$ in $\Ho(\DGA_k)$. Thus, $\gamma(f)=\gamma(g)$ in $\Ho(\DGA_k)$. (b) is checked similarly, using the natural analog of~\cite{BKR}, Proposition B.2 for $\mathtt{Poiss}$ (~\cite{BKR}, see the end of remark B.4.3).

\subsection{Remark}
\la{R11}
As mentioned in~\cite{BKR}, Section 5.6, we expect that the category $ \mathtt{NCPoiss}_k $ has a natural model structure compatible
with the standard (projective) model structure on $ \DGA_k $. The notation $\,\Ho^*\,$
is to remind the reader that, since $\mathtt{NCPoiss}_k $ is not yet proven to be a model category,
$\Ho^*(\mathtt{NCPoiss}_k)$ is not yet confirmed to be an abstract homotopy category in Quillen's sense.
One way to remedy this problem
is to use the construction of a fibre product (homotopy pullback) of model categories due to
To\"en (see \cite{To}). First, passing to infinite-dimensional limit $ V \to V_\infty $
(see \cite{BR}), we can stabilize the family of
representation functors replacing $ (\,\mbox{--}\,)_V^{\GL} $ in \eqref{di1} by
\begin{equation}
\la{reinf}
(\,\mbox{--}\,)_{\infty}^{\Tr}: \DGA_k \to  \cDGA_k\ .
\end{equation}
By~\cite{BR}, Theorem 4.2, \eqref{reinf} is a left Quillen functor having the total left derived functor
$\,\L(\,\mbox{--}\,)^{\Tr}_{\infty}\,$.
Then, by \cite{To}, the homotopy pullback of \eqref{reinf} along the forgetful functor $\,\mathtt{Poiss}_k \to \cDGA_k \,$
is a model category $\, \DGA_k \times^h_{\cDGA_k} \mathtt{Poiss}_k\,$, which, in view of \eqref{di1}, comes together
with a functor
$$
\mathtt{NCPoiss}_k \to \DGA_k \times^h_{\cDGA_k} \mathtt{Poiss}_k\ .
$$
It is easy to show that this last functor is homotopy invariant, so it induces a functor
\begin{equation}
\la{quieq}
\Ho^*(\mathtt{NCPoiss}_k) \to \Ho(\DGA_k \times^h_{\cDGA_k} \mathtt{Poiss}_k) \ .
\end{equation}
Our expectation is that \eqref{quieq} is an equivalence of categories.
Since To\"en's construction is known to give a correct notion of
`homotopy fibre product' (see \cite{Be}), such an equivalence would
mean that our {\it ad hoc} definition of NC Poisson structures is the
correct one from homotopical point view. This would also give a precise
meaning to the claim that the NC Poisson structures are the weakest structures
on $A$ inducing the usual Poisson structures under the representation functor
(since the fibre product $ \DGA_k \times^{h}_{\cDGA_k} \mathtt{Poiss}_k $ is
exactly the category that has this property).

\section{Noncommutative $P_{\infty}$-algebras}
\la{G.2}

The definition of an NC Poisson structure can be generalized to a definition of a NC $P_{\infty}$-structure. In this subsection, we show that a
NC $P_{\infty}$-structure on $A$ induces (in a functorial way) a $P_{\infty}$-structure on $A_V^{\GL}$ for all finite dimensional $V$.
Further, we show that homotopy equivalent NC $P_{\infty}$-structures induce homotopy equivalent $P_{\infty}$-structures on each $A_V^{\GL}$.
The main result in this subsection, i.e, Theorem~\ref{NCGinfRep}, is therefore, a stronger version of Theorem~\ref{NCPoiss}.
One can similarly define the notion of a NC $n$-$P_{\infty}$ structure on a DGA.
We remark here that Theorem~\ref{t3s2} holds word for word with NC $n$-Poisson replaced by NC $n$-$P_{\infty}$.
The reader who is interested only in NC Poisson structures may skip this section and move on the next section.

\subsection{Definitions}
\la{G.2.1}

A {\it $P_{\infty}$-structure} on a representation $V$ of $\DER(A)_{\natural}$ is a $L_{\infty}$-algebra structure
on $V$ whose adjoint representation\footnote{Recall that is $V$ is a $L_{\infty}$-algebra, one has a (structure) $L_{\infty}$-morphism $\text{ad}\,:\, V \rar
\End_k(V)$.} factors through $\varrho$: i.e, there is a $L_{\infty}$-morphism $i:V \rar \DER(A)_{\natural}$ (of $L_{\infty}$-algebras) such that $\varrho \circ i \,=\, \text{ad}$.

A {\it NC $P_{\infty}$-structure} on an object $A$ of $\DGA_k$ is, by definition, a NC $P_{\infty}$-
structure on $A_{\natural}$ (thought of as a representation of $\DER(A)_{\natural}$ as in Section~\ref{G.1.1}). Equivalently,
a NC $P_{\infty}$-algebra is a DG algebra $A$ such that $A_{\natural}$  is equipped with a
$L_{\infty}$-structure $\{l_n: \wedge^n (A_{\natural}) \rar (A_{\natural}) \}_{n \geq 1}$ such that $l_n$ has degree $n-2$ and
for all $a_1,..,a_{n-1}\,\in\,A$ homogenous, $l_n(\bar{a}_1,...,\bar{a}_{n-1},\mbox{--}):A_{\natural} \rar A_{\natural}$ is induced on
$A_{\natural}$ by a derivation $\partial_{a_1,...,a_{n-1}}:A \rar A\,\,\,$\footnote{ $\partial_{a_1,...,a_{n-1}}$ is any element of $\DER(A)$ such that
its image in $\DER(A)_{\natural}$ coincides with $i(\bar{a}_1 \wedge...\wedge \bar{a}_{n-1})$.} of degree $n-2+\sum_i |a_i|$.

A {\it morphism of NC $P_{\infty}$-algebras} is a collection of maps $f_1,\bar{f_2},..$ such that $f_1:A \rar B$ is a morphism of DG algebras
 and $\{\bar{f_n}:\wedge^n(A_{\natural}) \rar B_{\natural}\}_{n \geq 1}$ form an $L_{\infty}$-morphism \footnote{$\bar{f_1}:A_{\natural} \rar B_{\natural}$ is the map induced by $f_1$.}. We further require that for all
$a_1,..,a_{n-1} \in A$,
$\bar{f_n}(\bar{a}_1,...,\bar{a}_{n-1},-):A_{\natural} \rar B_{\natural}$ is induced by a degree $\sum |a_i|+n-1$ operator $\partial^f_{a_1,..,a_{n-1}}:
A \rar B$ such that the collection $\{\partial^f_{\textbf{a}_S}\}_{S \subset \{1,..,n-1\}}$ is a $f_1$-polyderivation of
 multi-degree $(|a_1|+1,..,|a_{n-1}|+1)$. Here, for $S=\{i_1<..<i_p\}$,
$\{\partial^f_{\textbf{a}_S}\}_{S}:=\partial^f_{a_{i_1},...,a_{i_p}}$. The reader is referred to~\cite{BKR}, Section 5.5
for the definition of polyderivations and related material.

In addition, note that if $B$ is a NC $P_{\infty}$-algebra and $\Omega$ is the de-Rham algebra of the affine line, then $B \otimes \Omega$ naturally acquires the structure of a NC $P_{\infty}$-algebra via extension of scalars.

This allows us to define the notion of homotopy along the lines of~\cite{BKR}, Proposition B.2 and Remark B.4.3. A homotopy between morphisms $f,g:A \rar B$ of $P_{\infty}$-algebras is a morphism $h:A \rar B \otimes
\Omega$ of $P_{\infty}$-algebras such that $h(0)=f,\,h(1)=g$.

We have thus, defined the category $\mathtt{NCP}_{\infty}$ of NC $P_{\infty}$ algebras. The full subcategory of $\mathtt{NCP}_{\infty}$
algebras whose objects are commutative DGAs equipped with NC $P_{\infty}$ structure will be called $\mathtt{P}_{\infty}$. We note
that our definition of $P_{\infty}$ algebra is more restrictive than the definition of $P_{\infty}$ in the operadic sense (see \cite{CVdb}
 for example). This definition, however, coincides with a definition of $\text{Poisson}_{\infty}$ algebras that has been studied earlier
in the literature (see \cite{CF, Sched09} for example).

The category $\Ho^*(\mathtt{NCP}_{\infty})$
 is defined to be the category whose objects are objects of $\mathtt{NCP}_{\infty}$ that are cofibrant in $\DGA_k$ such that
$\text{Hom}_{\Ho^*(\mathtt{NCP}_{\infty})}(A,B)$ is the space of homotopy equivalence classes in $\text{Hom}_{\mathtt{NCP}_{\infty}}(A,B)$.
The category $\mathtt{P}_{\infty}$ of $P_{\infty}$ algebras and its ``homotopy category'' are analogously defined on the commutative side.

\subsection{The main result} \la{G.2.2} The following result is the main result of this subsection. It strengthens Theorem~\ref{NCPoiss}.

\begin{theorem} \la{NCGinfRep}
(i) The functor $(-)_V^{\GL}:\DGA_k \rar \cDGA_k$ enriches to give the following commutative diagram of functors (vertical arrows being forgetful functors)
/$$\begin{diagram}
\mathtt{NCGP}_{\infty}& \rTo^{(-)_V^{\GL}} &\mathtt{P}_{\infty}\\
\dTo& &  \dTo\\
\DGA_k& \rTo^{(-)_V^{\GL}} &\cDGA_k
\end{diagram}$$
(ii) Further, the functor $(-)_V^{\GL}:\mathtt{NCP}_{\infty} \rar \mathtt{P}_{\infty}$ descends to a functor
$$\Ho^*{(-)_V^{\GL}}:\Ho^*(\mathtt{NCP}_{\infty}) \rar \Ho^*(\mathtt{P}_{\infty}) \text{.}$$
\end{theorem}

The proof of this theorem will be organized along the lines of the proof of Theorem~\ref{NCPoiss}.

\subsection{From $L_{\infty}$ to $P_{\infty}$ algebras}\la{G.2.3}

 The following proposition is (part of) Theorem 3.15 in \cite{Sched09}.

\begin{prop} \la{pG.2.2.1}
 Let $V$ be a $L_{\infty}$-algebra. Then, $\bSym(V)$ (with the obvious multiplication) inherits the structure of a $P_{\infty}$-algebra such that the
 $L_{\infty}$-operations on $\bSym(V)$ extend
those on $V$ via the formula
$$ l_n(v_1,.,v_{n-1},v.w)=l_n(v_1,..,v_{n-1},v).w+{(-1)}^{(\sum |v_i|+n-2)|v|}v.l_n(v_1,...,v_{n-1},w) \text{.}$$
\end{prop}

Recall the definition
of a polyderivation from~\cite{BKR}, Section 5.5. Let $\{f_n\}:V \rar G$ be a morphism of $L_{\infty}$-algebras (where $G$ is a $P_{\infty}$-algebra). Equip $\bSym(V)$ with the $P_{\infty}$-structures from Proposition~\ref{pG.2.2.1}. Each $f_n$ uniquely extends to a map
$$f_n:\wedge^n (\bSym(V)) \rar G$$ of degree $n-1$ such that for all $v_1,..,v_{n-1} \in \bSym(V)$,
the maps $\{f_S(v_1,..,v_{n-1},-):\bSym(V) \rar G\}_{S \subset \{1,...,n-1\}}$ constitute a polyderivation
of multi-degree $(|v_1|+1,..,|v_{n-1}|+1)$ with respect to $f_1:\bSym(V) \rar G$ \footnote{$f_1:\bSym(V) \rar G$
 is the obvious extension of the map of complexes $f_1:V \rar G$ to a morphism in $\cDGA_k$.}. Here, for $S=\{i_1<..<i_k\} \subset
\{1,...,n-1\}$, $f_S(v_1,..,v_{n-1},-)$ denotes the map $f_{k+1}(v_{i_1},..,v_{i_k},-):\bSym(V) \rar G$. The following key proposition is due to T.~Schedler (in~\cite{Schedpr}).

\begin{prop}
\la{pG.2.2.2}
The extended maps $\{f_n\}$ constitute the unique $P_{\infty}$-morphism from $\bSym(V)$ to $G$ extending $f:V \rar G$.
\end{prop}

We remark that Propositions~\ref{pG.2.2.1} and~\ref{pG.2.2.2} together imply that $V \mapsto \bSym(V)$ is a functor from
the category of $L_{\infty}$-algebras to the category of $P_{\infty}$-algebras. For a $L_{\infty}$-morphism $f:=\{f_n\}:V\rar G$, the unique $P_{\infty}$-morphism $\{f_n\}$ from $\bSym(V)$ to $G$ extending $f:V \rar G$ shall be denoted by $\bSym(f)$. Propositions~\ref{pG.2.2.1} and~\ref{pG.2.2.2} together imply that $\bSym((...))$ is a well defined functor from the category of $L_{\infty}$-algebras to the category of $P_{\infty}$-algebras. Recall that two $L_{\infty}$-morphisms $f,g:V \rar W$ are homotopic if there exists a $L_{\infty}$-morphism $h:V \rar W \otimes \Omega$ such that\footnote{The higher $L_{\infty}$-structure maps of $W \otimes \Omega$ are extended from those of $W$ by $\Omega$-linearity.} $h(0)=f$ and $h(1)=g$. One has the following lemma.

\begin{lemma} \la{pG.2.2.3}
 Suppose that $f,g:V \rar W$ are homotopic $L_{\infty}$-morphisms. Then, their extensions
$\bSym(f),\bSym(g):\bSym(V) \rar \bSym(W)$ are homotopic $P_{\infty}$-morphisms.
\end{lemma}

\begin{proof}
Let $h: V \rar W \otimes \Omega$ be a $L_{\infty}$-morphism with $h(0)=f$ and $h(1)=g$. The identification $\bSym(W) \otimes \Omega$
with $\bSym_{\Omega}(W \otimes \Omega)$ together with Proposition~\ref{pG.2.2.2} shows that $h$ extends to a $P_{\infty}$-morphism $\bSym(h):\bSym(V[) \rar \bSym(W) \otimes \Omega$ such that
$\bSym(h)(0)=\bSym(f):\bSym(V) \rar \bSym(W)$ and $\bSym(h)(1)=\bSym(g):\bSym(V) \rar \bSym(W)$.

\end{proof}

\subsection{Proof of Theorem~\ref{NCGinfRep}} \la{G.2.4}

\subsubsection{A technical lemma} Our proof requires the following technical lemma.

\begin{lemma} \la{pG.2.0}
Let $\{\phi_S\}$ be a polyderivation of multi-degree $(d_1,..,d_k)$ with respect to $\phi:A \rar B$. The induced polyderivation $\{(\phi_V)_S\}$ of
multi-degree $(d_1,..,d_k)$ from~\cite{BKR}, Lemma 5.5 restricts to a polyderivation $\{(\phi_V)_S\}$ with respect to $\phi_V:A_V^{\GL} \rar B_V^{\GL}$.
.
\end{lemma}
\begin{proof}
We need to verify that $(\phi_V)_S$ maps $A_V^{\GL}$ to $B_V^{\GL}$. For this, note that
$$ (\phi_V)_S(a.b)= \sum_{T \sqcup T'=S} \pm (\phi_V)_T(a).(\phi_V)_{T'}(b) \text{.}$$ The above equation implies that if $(\phi_V)_S(a)$ and $(\phi_V)_S(b)$ are in the subalgebra generated by the image of $(\mathrm{Tr}_V)_{\bullet}$
for all $S \subset \{1,..,k\}$, then $(\phi_V)_S(ab)$ is in the subalgebra generated by the image of $(\mathrm{Tr}_V)_{\bullet}$
for all $S \subset \{1,..,k\}$. This desired statement now follows from Theorem~\ref{t3s1}.
\end{proof}

\subsubsection{The main body of the proof}
It follows from Propositions~\ref{pG.2.2.1},~\ref{pG.2.2.2} and Lemma~\ref{pG.2.2.3} that if $f,g:A \rar B$ are homotopic morphisms
of NC $P_{\infty}$-algebras, $\bSym(f_{\natural})$ and $\bSym(g_{\natural})$ are homotopic morphisms of $P_{\infty}$-algebras
\footnote{A short explanation is needed here: indeed, $f$ consists of a morphism $f_1:A \rar B$ of DG-algebras and higher components
$\bar{f}_2,...,\bar{f}_n$ such that $\bar{f}_1,..,\bar{f}_n,...$ are Taylor components of a $L_{\infty}$-morphism from $A_{\natural}$
to $B_{\natural}$. The first component of $\bSym(f_{\natural})$ is $\bSym((f_1)_{\natural})$. The higher components are constructed
as in the discussion before Proposition~\ref{pG.2.2.2}. }. Let $l_n:\wedge^n(\bSym(A_{\natural})) \rar \bSym(A_{\natural})$ denote
the structure maps of $\bSym(A_{\natural})$.

\begin{lemma} \la{pG.2.3.1}
 For any $\beta\,\in\,\bSym(A_{\natural})$ such that $(\mathrm{Tr}_V)_{\bullet}(\beta)=0$,
$$ (\mathrm{Tr}_V)_{\bullet}(l_n(\alpha_1,...,\alpha_{n-1},\beta))=0 $$
for any $\alpha_1,..,\alpha_{n-1} \,\in\, \bSym(A_{\natural})$.
\end{lemma}

\begin{proof}
 Since
$$l_n(\alpha_1.\alpha_1',\alpha_2,..,\alpha_{n-1},\beta)= \pm \alpha_1.l_n(\alpha_1',...,\alpha_{n-1},\beta) \pm
\alpha_1'.l_n(\alpha_1,...,\alpha_{n-1},\beta)\,,$$
since $(\mathrm{Tr}_V)_{\bullet}\,:\,\bSym(A_{\natural}) \rar A_V^{\GL}$ is a ring homomorphism, and because the (anti)-symmetry of $l_n$,
an inductive argument reduces the verification the required lemma to
the case when $\alpha_1,...,\alpha_{n-1}\,\in\, A_{\natural}$. Let $\partial_{\alpha_1,...,
\alpha_{n-1}}$ be any derivation of $A$ whose image in $\der(A)_{\natural}=i(\alpha_1 \wedge...\wedge \alpha_{n-1})$. Then,
by Lemma~\ref{pG.2.0}, there is a (graded) derivation
$\psi_{\alpha_1,...,
\alpha_{n-1}}$ of $A_V^{\GL}$ such that
$$ (\mathrm{Tr}_V)_{\bullet}(\partial_{\alpha_1,...,
\alpha_{n-1}}(\tilde{\beta'}) \,=\, \psi_{\alpha_1,...,
\alpha_{n-1}}((\mathrm{Tr}_V)_{\bullet}(\beta'))$$ for any $\beta'\,\in\, A_{\sharp}$ and for any lift $\tilde{\beta'}$ of $\beta$ to $A$.
Hence,
$$  (\mathrm{Tr}_V)_{\bullet}(l_n(\alpha_1,...,\alpha_{n-1},\beta'))\,=\, \psi_{\alpha_1,...,
\alpha_{n-1}}((\mathrm{Tr}_V)_{\bullet}(\beta'))$$
for any $\beta' \in A_{\sharp}$. It follows that
$$ (\mathrm{Tr}_V)_{\bullet}(l_n(\alpha_1,...,\alpha_{n-1},\beta))\,=\, \psi_{\alpha_1,...,
\alpha_{n-1}}((\mathrm{Tr}_V)_{\bullet}(\beta))$$ for any $\beta \,\in\, \bSym(A_{\natural})$. The right hand side of the above equation
is clearly $0$ if $(\mathrm{Tr}_V)_{\bullet}(\beta)=0$.

\end{proof}

By Lemma~\ref{pG.2.3.1} and Theorem~\ref{t3s1}, the operations $\bar{l}_n:\wedge^n(A_V^{\GL}) \rar A_V^{\GL}$ defined by
$$ \bar{l}_n( \alpha_1,..,\alpha_n) \,=\, (\mathrm{Tr}_V)_{\bullet}(l_n((\mathrm{Tr}_V)_{\bullet}^{-1}(\alpha_1),...,(\mathrm{Tr}_V)_{\bullet}^{-1}
(\alpha_n)))$$
 are well defined. Further, Theorem~\ref{t3s1} and the fact that the $l_n$ impose a $P_{\infty}$ structure on
$\bSym(A_{\natural})$ together imply that the $\bar{l}_n$ impose a $P_{\infty}$-structure on $A_V^{\GL}$.
Let $\{f_1,\bar{f}_2,...,\}\,:\, A \rar B$ be a homomorphism of NC $P_{\infty}$-algebras.

\begin{lemma} \la{pG.2.3.2}
As in Proposition~\ref{pG.2.2.2}, extend the $\bar{f}_n$ to maps $f_n:\wedge^n(\bSym(A_{\natural})) \rar \bSym(B_{\natural})$.
Then, for any $\beta \in \bSym(A_{\natural})$ such that $(\mathrm{Tr}_V)_{\bullet}(\beta)=0$,
$$ (\mathrm{Tr}_V)_{\bullet}(f_n(\alpha_1,...,\alpha_{n-1},\beta)) = 0$$
for all $\alpha_1,..,\alpha_{n-1} \,\in\, \bSym(A_{\natural})$.

\end{lemma}

\begin{proof}
 For $n=1$, the required lemma is clear: indeed, $(\mathrm{Tr}_V)_{\bullet}$ is a natural transformation between the functors
$A \mapsto \bSym(A_{\natural})$ and $A \mapsto A_V^{\GL}$. Put $\alpha_n:=\beta$. Note that $\bSym(A_{\natural})$ has a polynomial
grading where elements of $A_{\natural}$ may be viewed as the elements of polynomial degree $1$. Put $J_n:=\{2,3,...,n\}$.
For $I:=\{\alpha_{i_1}<...<\alpha_{i_{|I|}}\}\, \subset\, J_n$ let
 $f_I(\alpha):=f_{|I|+1}(\alpha,\alpha_{i_1},..,\alpha_{i_{|I|}})$.
Since
$$ f_n(\alpha_1.\alpha_1',\alpha_2,..,\alpha_n) \,=\, \sum_{I \subset J_n}  \pm f_{I}(\alpha).f_{J \setminus I}(\alpha')\,,$$
an induction on $n$ as well as the polynomial degree of $\alpha_1$, together with the (anti)-symmetry of $f_n$ reduces the verification of
the required lemma to the case when $\alpha_1,..,\alpha_{n-1} \,\in\, A_{\natural}$. In this case, let $a_i$ be a lift of
$\alpha_i$ for $1 \leq i \leq n-1$. Consider the polyderivation $\{\partial^f_{\textbf{a}_S}\}_{S \subset \{1,..,n-1\}}$ from
Section~\ref{G.2.1}. By Lemma~\ref{pG.2.0}, there is a polyderivation $\{\psi^f_{\textbf{a}_S}:A_V^{\GL} \rar B_V^{\GL}\}_{S \subset \{1,..,n-1\}}$
 satisfying
$$ (\mathrm{Tr}_V)_{\bullet}(\partial^f_{\textbf{a}_S}(b)) \,=\, \psi^f_{\textbf{a}_S}((\mathrm{Tr}_V)_{\bullet}(b)) $$
for any $b \in A$, $S \subset \{1,...,n-1\}$. Hence,
$$ (\mathrm{Tr}_V)_{\bullet}(f_{|S|+1}(\alpha_{i_1},..,\alpha_{i_{|S|}},\beta'))\,=\, \psi^f_{\textbf{a}_S}((\mathrm{Tr}_V)_{\bullet}(\beta'))$$
for all $\beta' \in A_{\natural}$ and $S:=\{i_1<...<i_{|S|}\} \subset \{1,...,n-1\}$. Since the operators \\
$\{(\mathrm{Tr}_V)_{\bullet}(f_{|S|+1}(\alpha_{i_1},..,\alpha_{i_{|S|}},-))\}_S$ as well as $\{\psi^f_{\textbf{a}_S}((\mathrm{Tr}_V)_{\bullet}(-))\}_S$
constitute polyderivations of the same multidegree with respect to $(\mathrm{Tr}_V)_{\bullet} \circ \bSym((f_1)_{\natural})$, it follows that
$$ (\mathrm{Tr}_V)_{\bullet}(f_{|S|+1}(\alpha_{i_1},..,\alpha_{i_{|S|}},\beta))\,=\, \psi^f_{\textbf{a}_S}((\mathrm{Tr}_V)_{\bullet}(\beta))$$
for all $\beta \in \bSym(A_{\natural})$. Since the right hand side of the above equation vanishes when $(\mathrm{Tr}_V)_{\bullet}(\beta)=0$,
the required lemma follows.
\end{proof}

Lemma~\ref{pG.2.3.2} and Theorem~\ref{t3s1} imply that the maps
$$\wedge^n(A_V^{\GL}) \rar B_V^{\GL},\,\,\,\,\, (\alpha_1,..,\alpha_n) \mapsto  (\mathrm{Tr}_V)_{\bullet}(f_n( (\mathrm{Tr}_V)_{\bullet}^{-1}(\alpha_1),...,
 (\mathrm{Tr}_V)_{\bullet}^{-1}(\alpha_n))) $$
are well defined.  Theorem~\ref{t3s1} together with Proposition~\ref{pG.2.2.2} further implies that they
constitute a $P_{\infty}$-morphism $f_V^{\GL}$ from $A_V^{\GL}$ to $B_V^{\GL}$ (and that this construction is preserves compositions of morphisms).
This proves Theorem~\ref{NCGinfRep} (i). A trivial modification of the same argument shows that if $h:A \rar B \otimes \Omega$ is a
P-homotopy between morphisms $f,g:A \rar B \otimes \Omega$ of NC $P_{\infty}$-algebras, then the above construction
applied to $h$ gives a morphism $h_V^{\GL}:A_V^{\GL} \rar B_V^{\GL}$ that is a homotopy between the morphisms $f_V^{\GL}$ and $g_V^{\GL}$
of $P_{\infty}$-algebras. This proves Theorem~\ref{NCGinfRep} (ii).

\subsection{Remark}

Another interesting question would be to relate $\Ho^*(\mathtt{NCP}_{\infty})$ to $\Ho^*(\mathtt{NCPoiss})$. In this context, one
may recall a theorem due to Munkholm (\cite{M}), which states that $\Ho(\DGA_k)$ is equivalent to the full subcategory
of $\Ho^*(A_{\infty})$ whose objects are objects of $\DGA_k$. Here, $\Ho^*(A_{\infty})$ is the category whose objects are $A_{\infty}$-algebras
 and whose morphisms are homotopy classes of $A_{\infty}$-morphisms. Given this, one may ask whether $\Ho^*(\mathtt{NCPoiss})$
is equivalent to the full subcategory of $\Ho^*(\mathtt{NCP}_{\infty})$ whose objects are objects of $\mathtt{NCPoiss}$. If the answer
to this question is negative, what is the precise relation between these two categories ?

\section{NC Poisson algebras from Calabi-Yau algebras} \la{s3}

This section exhibits a large family of NC $n$-Poisson algebras: more generally, we show that if $A$ is a finite dimensional graded $n$-cyclic algebra, the cobar construction applied to the linear dual $C:=\Hom_k(A,k)$ of $A$ is equipped with a $(2-n)$-Poisson double bracket. This Poisson double bracket induces a noncommutative $(2-n)$-Poisson structure. Of course, the cobar construction of $C$ is cofibrant (in fact, free) as a DG algebra. In particular, a finite dimensional $2$-cyclic algebra (for example, the cohomology of a compact smooth $2$-manifold) gives rise to a cofibrant DG algebra with a noncommutative Poisson structure. We point out that a $n$-cyclic algebra is a special case of a $n$-Calabi-Yau $A_{\infty}$-algebra in the sense of Kontsevich (see~\cite{Cos}, Section 7.2).

In this section, we shall often use Sweedler's notation and write
$$ \Delta(\alpha)\,=\, \alpha' \otimes \alpha''$$
for any element $\alpha$ in a coalgebra $C$ with coproduct $\Delta$.

\subsection{Double Poisson (DG) algebras}

Let $A$ be an associative DG algebra over a field $k$.
An $A$-bimodule $M$ is a left $A\otimes A^{\mathrm{op}}$-module.
On $A\otimes A$, there are two $A$-bimodule structures:
one is the outer $A$-bimodule, namely
$$a\cdot(u\otimes v)\cdot b=au\otimes vb;$$
the other one is the inner $A$-bimodule, namely
$$a\cdot (u\otimes v)\cdot b= {(-1)}^{|a||u|+|a||b|+|b||v|}ub\otimes av\,\text{.}$$
Here, $a,b,u,v$ are arbitrary homogenous elements of $A$.

Suppose that $A$ is an associative (unital) DG algebra over a field $k$.
A {\it double bracket} of degree $n$ on $A$ is a bilinear map
$$\ldb-,-\rdb:A\otimes A\to A\otimes A$$
which is a derivation of degree $n$ (for the outer $A$-bimodule structure on $A\otimes A$) in its second argument
and satisfies
$$\ldb a,b\rdb=-{(-1)}^{(|a|+n)(|b|+n)}\ldb b,a\rdb^\circ,$$
where $(u\otimes v)^\circ ={(-1)}^{|u||v|}v\otimes u$.
For $a,b_1,...,b_n$ homogeneous in $A$, let
$$\ldb a, b_1 \otimes \ldots \otimes b_n\rdb_L \,:=\, \ldb a,b_1 \rdb \otimes b_2 \otimes \ldots \otimes b_n \,\text{.}$$
Further, for a permutation $s \in S_n$, let
$$\sigma_s(b_1 \otimes \ldots \otimes b_n)\,:=\, {(-1)}^t b_{s^{-1}(1)} \otimes \ldots \otimes b_{s^{-1}(n)}\,$$
where
$$ t:=\sum_{i<j; s^{-1}(i)>s^{-1}(j)} |a_{s^{-1}(i)}||a_{s^{-1}(j)}| \,\text{.}$$

Suppose that $\ldb\mbox{--},\mbox{--}\rdb$ is a double bracket of degree $n$ on $A$. If furthermore $A$ satisfies the following {\it double Jacobi identity}
$$ \ldbg a , \ldb b,c \rdb \rdbg_L + {(-1)}^{(|a|+n)(|b|+|c|)}\sigma_{(123)}\ldbg b,\ldb c,a\rdb \rdbg_L + {(-1)}^{(|c|+n)(|a|+|b|)} \sigma_{(132)}
\ldbg c,\ldb a,b\rdb \rdbg_L =0\,,$$
then $A$ is called a {\it double $n$-Poisson algebra}.

Let $\mu\,:\,A \otimes A \rar A$ denote the multiplication on $A$. Let $\{\mbox{--},\mbox{--}\}\,:=\, \mu \circ \ldb\mbox{--},\mbox{--}\rdb\,:\,A\otimes A \rar A$. The following is a direct generalization (to the DG setting) of Lemma 2.4.1 of~\cite{vdB}.

\begin{lemma}\la{l1s3}

$\{\mbox{--},\mbox{--}\}\,$ induces a noncommutative $n$-Poisson structure on $A$. In particular, when $n=1$, $\{\mbox{--},\mbox{--}\}\,$ induces a noncommutative Gerstenhaber structure on $A$.
\end{lemma}

\noindent
{\bf Notation.} Using Sweedler's notation, we will often write
$$\ldb u,v\rdb\,=\, \ldb u,v\rdb'\otimes \ldb u,v\rdb'' \,\text{.}$$

\subsubsection{}
\textbf{Remark.} One thus, has the category $n-\mathtt{DPoiss}$ of DG double $n$-Poisson algebras.
We sketch how the analog of Theorem~\ref{NCPoiss} holds for double $n$-Poisson algebras. The details in this subsubsection are left to the interested reader. Let $A$ be a DG $n$-double Poisson algebra.
Let $\Omega$ denote the de Rham algebra of the affine line, as in Section~\ref{G}.
Then, it is verified without difficulty
that $A \otimes \Omega$ has a $\Omega$-linear DG double $n$-Poisson structure. Indeed, after identifying
$(A \otimes \Omega) \otimes_{\Omega} (A \otimes \Omega)$ with $A \otimes A \otimes \Omega$,
the ($\Omega$-linear) double bracket on $A \otimes \Omega$ is given by $\ldb \mbox{--},\mbox{--} \rdb \otimes \id_{\Omega}$.

One may therefore, define the ``homotopy category" $\Ho^*(n-\mathtt{DPoiss})$ of $n-\mathtt{DPoiss}$:
objects in $\Ho^*(n-\mathtt{DPoiss})$ are objects in $n-\mathtt{Poiss}$ that are cofibrant in $\DGA_k$.
Morphisms in $\Ho^*(n-\mathtt{DPoiss})$ are homotopy classes of morphisms in $n-\mathtt{Poiss}$.
Here, $f,g\,:\,A \rar B$ in $n-\mathtt{Poiss}$ are {\it homotopic} if there
exists $h\,:\,A \rar B\otimes \Omega$ in $n-\mathtt{Poiss}$ such that $h(0)=f$ and $h(1)=g$. Here, when we say that
$h\,:\,A \rar B \otimes \Omega$ is in $n-\mathtt{Poiss}$, we mean that
$$\ldb \mbox{--},\mbox{--} \rdb \circ (h \otimes_{\Omega} h) \,=\, (h \otimes_{\Omega} h) \circ \ldb \mbox{--},\mbox{--} \rdb \,\text{.} $$

If $A$ is a double $n$-Poisson algebra,
a direct extension of the proofs of Propositions 7.5.1 and 7.5.2 of~\cite{vdB} shows that there exists a
DG $n$-Poisson structure on $A_V$ (which restricts to the one induced on $A_V^{\GL}$ by the corresponding
noncommutative $n$-Poisson structure on $A$). One further shows without much difficulty that if $A,B$ are
double $n$-Poisson algebras and if $h\,:\,A \rar B\otimes \Omega$ is a morphism of double $n$-Poisson algebras,
then $h_V\,:\,A_V \rar B_V \otimes \Omega$ is a morphism of DG $n$-Poisson algebras. Thus, the analog of
Theorem~\ref{NCPoiss} (with $(\mbox{--})_V$ and $\L(\mbox{--})_V$ replacing
$(\mbox{--})_V^{\GL}$ and $\L(\mbox{--})_V^{\GL}$) holds for double $n$-Poisson algebras.

\subsection{Cyclic graded algebras and double Poisson brackets}
In this subsection, we will avoid specifying exact signs that are determined by the Koszul rule. Instead, such signs will be denoted by the symbol $\pm$. This is done in order to simplify cumbersome formulas, especially in the proof of Theorem~\ref{t1s3}.

\subsubsection{}
Let $A$ be a finite dimensional (graded) associative algebra with a symmetric inner product of degree $n$ such that
\begin{equation}\label{cyclicpairing}
\langle a,bc\rangle=\pm \langle ca,b\rangle,\quad\mbox{for any}\; a,b,c\in A.
\end{equation}
According to Kontsevich (\cite{Ko}) and Getzler-Kapranov (\cite{GeKa}),
such an algebra is called a $n$-{\it cyclic} associative algebra.
In addition, if $A$ is finite dimensional, the dual space $C:=\mathrm{Hom}(A,k)$
is a coalgebra equipped with a symmetric bilinear pairing of degree $-n$. By the non degeneracy of the inner product on $A$,
Equation (\ref{cyclicpairing}) is dual to the following identity:
\begin{equation}\label{coprod}
\langle v',w\rangle \cdot v''\,=\, \pm \langle v,w''\rangle\cdot w',
\quad\mbox{for any}\; v,w\in C.
\end{equation}
Hence, $C$ acquires the structure of {\it cyclic} $(-n)$-coalgebra. More generally, a DG coalgebra
$C$ equipped with a symmetric bilinear pairing $\langle\mbox{--},\mbox{--}\rangle$ of degree $n$ is called
{\it cyclic} if in addition to~\eqref{coprod},
\begin{equation}
\la{idif} \langle du,v \rangle \pm \langle u, dv \rangle = 0
\end{equation}
for all $u,v \in C$.

\subsubsection{Constructing the $(n+2)$-double Poisson bracket}

Let $C$ be a $n$-cyclic coassociative coalgebra and let $\B(C)$ denote
the cobar construction of $C$. Define  $\,\ldb-,-\rdb:\B(C)\otimes \B(C)\to\B(C)\otimes \B(C)\,$
by
\begin{equation}
\label{def_bb}
\ldb v, w\rdb := \displaystyle
\sum_{i=1,\cdots,n\atop j=1,\cdots,m}\pm \langle v_i,w_j\rangle\cdot
(w_1,\cdots,w_{j-1},v_{i+1},\cdots,v_n)\,\otimes\,
(v_1,\cdots,v_{i-1},w_{j+1},\cdots,w_m)\, ,
\end{equation}
where $ v = (v_1,v_2,\cdots,v_n) $ and $ w = (w_1,w_2,\cdots,w_m) $. The next theorem
is the main result of this section.
\begin{theorem}
\label{t1s3}
Let $C$ be a $n$-cyclic coassociative DG coalgebra. The bracket \eqref{def_bb}
gives a double $(n+2)$-Poisson structure on the DG algebra $\B(C)$.
\end{theorem}
\begin{remark}
By construction, $\B(C)$ is a cofibrant DG algebra.
\end{remark}

\begin{proof}
 The proof consists of three steps.

{\bf Step 1.} First, recall that $\B(C)$ has a natural differential graded algebra structure,
with multiplication given by the tensor product. We show that $\ldb-,-\rdb$
is a derivation for the second argument.
For $\, u=(u_1,u_2,\cdots,u_p),\, v=(v_1,v_2,\cdots,v_q),\, w=(w_1,w_2,\cdots,w_r)\,$,
we have
\begin{eqnarray*}
\ldb u,v\cdot w\rdb
&=&\ldb (u_1,\cdots,u_p),(v_1,\cdots,v_q,w_1,\cdots,w_r)\rdb\\*[2ex]
&=&\sum_{i=1,\cdots,p\atop j=1,\cdots,q} \pm \langle u_i,v_j\rangle\cdot
(v_1,\cdots,v_{j-1},u_{i+1},\cdots, u_p)\otimes(u_1,\cdots,u_{i-1},v_{j+1},\cdots,v_q,w_1,\cdots,w_r)\\
&+&\sum_{i=1,\cdots,p\atop k=1,\cdots,r}\pm \langle u_i,w_k\rangle\cdot
(v_1,\cdots,v_q,w_1,\cdots,w_{k-1},u_{i+1},\cdots,u_p)\otimes(u_1,\cdots,u_{i-1},w_{k+1},\cdots,w_r).
\end{eqnarray*}
Hence
\begin{equation} \la{dder} \ldb u,v\cdot w\rdb \,=\,  \ldb u,v\rdb'\otimes\ldb u,v\rdb''\cdot w \pm v\cdot\ldb u,w\rdb'\otimes\ldb u,w\rdb''\,\text{.}
 \end{equation}

{\bf Step 2.} Next, we show that $\ldb-,-\rdb $ is skew symmetric and satisfies the double Jacobi
identity. The skew symmetricity follows directly from the definition \eqref{def_bb}
as the pairing on $C[-1]$ induced by $\langle\mbox{--},\mbox{--}\rangle$ is skew-symmetric.
Therefore, we only need to check the double Jacobi identity.
For $u=(u_1,u_2,\cdots,u_p),\, v=(v_1,v_2,\cdots,v_q),\, w=(w_1,w_2,\cdots,w_r)$, we have
\begin{eqnarray*}
\ldb u,v\rdb&=&\sum_{i,j}\pm \langle u_i,v_j\rangle\cdot
(v_1,\cdots,v_{j-1},u_{i+1},\cdots, u_p)\otimes(u_1,\cdots, u_{i-1},v_{j+1},\cdots,v_q)\ ,\\
\ldb v,w\rdb&=&\sum_{j,k}\pm \langle v_j,w_k\rangle \cdot
(w_1,\cdots,w_{k-1},v_{j+1},\cdots, v_q)\otimes(v_1,\cdots, v_{j-1},w_{k+1},\cdots,w_r)\ ,\\
\ldb w,u\rdb&=&\sum_{k,i}\pm \langle w_k,u_i\rangle \cdot
(u_1,\cdots,u_{i-1},w_{k+1},\cdots, w_r)\otimes(w_1,\cdots, w_{k-1},u_{i+1},\cdots,u_p)\ .
\end{eqnarray*}
Therefore
\begin{eqnarray}
&&\ldbg u,\ldb v,w\rdb'\rdbg\otimes\ldb v,w\rdb'' = \nonumber\\*[2ex]
&&\sum_{i,j,k\atop 1\le l\le k-1}\pm \langle v_j, w_k\rangle\langle u_i,w_l\rangle\cdot
(w_1,\cdots,w_{l-1},u_{i+1},\cdots,u_p)\nonumber\\
&&\quad\quad\quad\quad\otimes(u_1,\cdots,u_{i-1},w_{l+1},\cdots,w_{k-1},v_{j+1},\cdots,v_q)
\otimes(v_1,\cdots,v_{j-1},w_{k+1},\cdots,w_r)\label{sum_I} + \\*[2ex]
&&
\sum_{i,j,k\atop j+1\le m\le q}\pm \langle v_j,w_k\rangle\langle u_i,v_m\rangle\cdot
(w_1,\cdots,w_{k-1},v_{j+1},\cdots,v_{m-1},u_{i+1},\cdots,u_p)\nonumber\\
&&\quad\quad\quad\quad\otimes (u_1,\cdots, u_{i-1},v_{m+1},\cdots,v_q)
\otimes(v_1,\cdots,v_{j-1},w_{k+1},\cdots,w_r),
\label{sum_II}\\*[2ex]
&&\ldb w,u\rdb''\otimes\ldbg v,\ldb w,u\rdb'\rdbg = \nonumber\\
&&
\sum_{i,j,k\atop 1\le t\le i-1}\pm \langle w_k,u_i\rangle\langle v_j,u_t\rangle\cdot
(w_1,\cdots,w_{k-1},u_{i+1},\cdots,u_p)\nonumber\\
&&\quad\quad\quad\quad \otimes(u_1,\cdots,u_{t-1},v_{j+1},\cdots,v_q)\otimes
(v_1,\cdots, v_{j-1},u_{t+1},\cdots,u_{i-1},w_{k+1},\cdots,w_r)\label{sum_III} + \\*[2ex]
&&
\sum_{i,j,k\atop k+1\le s\le n}
\pm \langle w_k,u_i\rangle\langle v_j,w_s\rangle\cdot
(w_1,\cdots,w_{k-1},u_{i+1},\cdots,u_p)\nonumber\\
&&\quad\quad\quad\quad
\otimes(u_1,\cdots,u_{i-1},w_{k+1},\cdots,w_{s-1},v_{j+1},\cdots,v_q)\otimes
(v_1,\cdots,v_{j-1},w_{s+1},\cdots,w_n),\label{sum_IV}\\*[2ex]
&&\ldbg w,\ldb u,v\rdb'\rdbg''\otimes\ldb u,v\rdb''\otimes
\ldbg w,\ldb u,v\rdb'\rdbg' = \nonumber\\
&&
\sum_{i,j,k\atop 1\le n\le j-1}
\pm \langle u_i,v_j\rangle\langle w_k,v_n\rangle\cdot
(w_1,\cdots,w_{k-1},v_{n+1},\cdots,v_{j-1},u_{i+1},\dots,u_p)\nonumber\\
&&\quad\quad\quad\quad\otimes(u_1,\cdots,u_{i-1}, v_{j+1},\cdots,v_q)
\otimes (v_1,\cdots,v_{n-1},w_{k+1},\cdots,w_r)\label{sum_V}+\\*[2ex]
&&
\sum_{i,j,k\atop j+1\le m\le p}\pm \langle u_i,v_j\rangle\langle w_k, u_m\rangle\cdot
(w_1,\cdots,w_{k-1},u_{m+1},\cdots,u_p)\nonumber\\
&&\quad\quad\quad\quad
\otimes(u_1,\cdots,u_{i-1},v_{j+1},\cdots,v_q)\otimes(v_1,\cdots,v_{j-1},u_{i+1},\cdots,u_{m-1},w_{k+1},\cdots,
w_r)\label{sum_VI}
\end{eqnarray}
In the above equations, the summand (\ref{sum_I}) cancels with (\ref{sum_IV}),
(\ref{sum_II}) cancels with (\ref{sum_V}), and (\ref{sum_III}) cancels with (\ref{sum_VI}).
So we get
$$
\ldbg u,\ldb v,w\rdb'\rdbg\otimes\ldb v,w\rdb''\pm
\ldb w,u\rdb''\otimes\ldbg v,\ldb w,u\rdb'\rdbg \pm \ldbg w,\ldb u,v\rdb'\rdbg''\otimes
\ldb u,v\rdb''\otimes\ldbg w,\ldb u,v\rdb'\rdbg'=0,
$$
which proves the double Jacobi identity.

{\bf Step 3.} Using equation~\eqref{dder}, one verifies without difficulty that if
$\partial \ldb u,v \rdb \,=\, \ldb \partial u,v \rdb \pm \ldb u,\partial v \rdb$ and if
$\partial \ldb u,w \rdb \,=\, \ldb \partial u,w \rdb \pm \ldb u,\partial w \rdb$, then
$\partial \ldb u,vw \rdb \,=\, \ldb \partial u,vw \rdb \pm \ldb u,\partial(vw) \rdb$. By this fact and the skew-symmetry
of $\ldb\mbox{--},\mbox{--}\rdb$, it suffices to verify that
$\partial \ldb u,v \rdb \,=\, \ldb \partial u,v \rdb \pm \ldb u,\partial v \rdb$ for all $u,v \in C$. In this case,
$\partial \ldb u,v \rdb=0$. On the other hand, $\partial u= du \pm (u',u'')$ and $\partial v= dv \pm (v',v'')$.
Hence,
\begin{equation*}
\ldb \partial u,v \rdb \pm \ldb u,\partial v \rdb
= (\langle du, v \rangle \pm \langle u, dv \rangle) +
(\langle u', v\rangle u'' \pm \langle u, v''\rangle v' \pm \langle v, u''\rangle u' \pm \langle v', u\rangle v'')
\ ,
\end{equation*}
where the first parenthesis in the right hand side vanishes by~\eqref{idif} and the second by~\eqref{coprod}.
This proves that $\partial \ldb u,v \rdb \,=\, \ldb \partial u,v \rdb \pm \ldb u,\partial v \rdb$ for arbitrary
$u,v \in \B(C)$, completing the proof of the theorem.
\end{proof}

\subsubsection{} \textbf{Remark}. One say that two morphisms $f,g\,:\, C_1 \rar C_2$ of $n$-cyclic coalgebras are {\it homotopic}
if there exists a family $\phi_t\,:\, C_1 \rar C_2$ of morphisms of $n$-cyclic coalgebras varying polynomially with $t$
as well as degree $1$ coderivations $s_t$ with respect to $\phi_t$ such that
$$
\phi_0=f, \phi_1=g \text{ and } \frac{d\phi_t}{dt}\,=\, [d,s_t]\,\text{.}
$$
We further require that for all $u,v \in C_1$,
$$
\langle s_t(u),\phi_t(v)\rangle \pm \langle \phi_t(u),s_t(v)\rangle =0\,\text{.}\
$$
The above notion of homotopy is dual to the notion of a polynomial M-homotopy betwee two morphisms in $\DGA_k$
(see~\cite{BKR}, Proposition B.2 and subsequent remarks). Extending Theorem~\ref{t1s3}, one can further show that
if  $f,g\,:\, C_1 \rar C_2$  are homotopic as morphisms of $n$-cyclic coalgebras, $\B(f)$ is homotopic to $\B(g)$
as morphisms of $(n+2)$-double Poisson algebras. We leave the relevant details to the motivated reader.

\subsection{Cyclic homology of coalgebras}
We recall the definition of Hochschild and cyclic homology of coalgebras.
Given a coalgebra $C$ over $k$ consider the following double complex which is obtained
by reversing the arrows in the standard (Tsygan) double complex of an algebra:
$$
\xymatrixcolsep{4pc}
\xymatrix{
& & & & \\
C^{\otimes 3}\ar[u]_b\ar[r]^{1-T} &C^{\otimes 3}\ar[u]_{b'}\ar[r]^N &C^{\otimes 3}\ar[u]_b\ar[r]^{1-T}&\\
C^{\otimes 2}\ar[u]_b\ar[r]^{1-T} &C^{\otimes 2}\ar[u]_{b'}\ar[r]^N &C^{\otimes 2}\ar[u]_b\ar[r]^{1-T}&\\
C\ar[u]_b\ar[r]^{1-T} &C \ar[u]_{b'}\ar[r]^N &C \ar[u]_b\ar[r]^{1-T}&\\
0\ar[u]&0\ar[u]&0\ar[u]&
}
$$
This double complex is $2$-periodic in horizontal direction, with operators
$ b, b', T$ and $N$ given by 
\begin{eqnarray*}
b'(c_1,\cdots,c_n)&=&\sum_{i=1}^{n-1}(-1)^{i-1}(c_1,\cdots,c_i',c_{i+1}',\cdots,c_n)\ ,\\
b(c_1,\cdots,c_n)&=&b'(c_1,\cdots,c_n)+\sum (-1)^{n}(c_1'',c_2,\cdots,c_{n},c_1')\ ,\\
T(c_1,\cdots,c_n)&=&(-1)^{n-1}(c_2,\cdots,c_n,c_1)\ ,\\
N=\sum_{i=0}^{n-1}T^i\ .&&
\end{eqnarray*}
The $b$-column is called the Hochschild chain complex $\,\CH_{\bullet}(C,C)$ of $C\,$:
it defines the {\it Hochschild homology} $\,\HH_{\bullet}(C)$. The kernel of $1-T$ from
the $b$-complex to the $b'$-complex is called the cyclic complex $\CC_{\bullet}(C)$:
by definitiion, its homology is the {\it cyclic homology} $\HC_{\bullet}(C)$ of $C$.

\vspace{2ex}

\begin{remark}\label{relations_algebracoalgebra}
If $C$ is a coalgebra, then the dual complex $A:=\mathrm{Hom}(C,k)$ admits
an algebra structure.
If furthermore $C$ is finite dimensional, then the Hochschild complex $\CH_{\bullet}(C,C)$ (resp., cyclic complex $\CC_*(C)$)
is isomorphic to the Hochschild cochain complex $\CH^{\bullet}(A,k)$ (reps., cyclic cochain complex $\CC^{\bullet}(A)$).
Here the Hochschild cochain complex $\CH^{\bullet}(A,k)$ is the Hochschild cochain complex of $A$ {\it with values in} $k$.
Otherwise if $C$ is infinite dimensional, then the Hochschild complex
$\CH_{\bullet}(C,C)$ (reps. cyclic complex $\CC_{\bullet}(C)$)
is a sub complex of the Hochschild cochain complex $\CH^{\bullet}(A,k)$ (reps. cyclic cochain complex $\CC^{\bullet}(A)$).\\
\end{remark}

We collect some facts about the cyclic complex from Quillen \cite[\S1.3]{Q3}.
Let $A$ be an associative algebra. The commutator subspace of $A$ is $[A,A]$, which is the image
of $m-m\sigma: A\otimes A\to A$, where $m$ is the product and $\sigma$ is the switching
operator, and the commutator quotient space
is
$$A_\natural:=A/[A,A]=\mathrm{Coker}\{m-m\sigma:A\otimes A\to A\}.$$
Dually, suppose $C$ is a coassociative coalgebra,
the cocommutator sub space of $C$
is
$$C^\natural:=\mathrm{Ker}\{\Delta-\sigma\Delta:C\to C\otimes C\}.$$
Recall that the bar construction $\mathbf B(A)$ of $A$ (resp. cobar construction
$\mathbf \Omega(C)$ of $C$) is a differential graded (DG) coalgebra
(resp. DG algebra). The following lemma is~\cite[Lemma 1.2]{Q3}.

\begin{lemma}
The space $\mathbf B_n^\natural (A)$ is the kernel of $(1-T)$ acting
on $A^{\otimes n}$.
\end{lemma}

Dually, the space $\mathbf (\Omega_\natural(C))_n$ is the cokernel of $(1-T)$ acting on $C^{\otimes n}$.
And therefore, via the isomorphisms
$$\CC_{\bullet}(A)=\mathrm{Coker}(1-T)\stackrel{\cong}{\to}
\mathrm{Ker}(1-T),\quad
\CC_{\bullet}(C)=\mathrm{Ker}(1-T)\stackrel{\cong}{\to}
\mathrm{Coker}(1-T),$$
one obtains the following lemma.

\begin{lemma}\la{l2s3}
As complexes of $k$-vector spaces,
$$\Omega(C)_{\natural} \,\cong\,\CC_{\bullet}(C)\,\text{.}$$
\end{lemma}

All explicit examples of derived NC Poisson structures in this paper arise by applying Theorem~\ref{t1s3} and Lemma~\ref{l1s3} to a cofibrant resolution of an honest algebra of $A$ that is of the form $\Omega(C)$ for some finite dimensional cyclic DG coalgebra $C$. This makes the following corollary of this paper relevant.

\begin{corollary} \la{c2s3}
Let $C$ be a (DG) coalgebra such that $\Omega(C) \stackrel{\sim}{\twoheadrightarrow} A$ in $\DGA_k$ for some $A \in \Alg_k $. Then,
$$\mathrm{HC}_{\bullet}(A)\,\cong\,\mathrm{HC}_{\bullet}(C) \,\text{.}$$

\end{corollary}

 Now suppose that $A$ is a finite dimensional graded $k$-algebra. By definition, $\mathrm{HC}^i(A)\,\cong\,\mathrm{HC}_i(A)^{\ast}$. It is concentrated in homological degree $-i$.  By Theorem~\ref{t1s3}, Lemma~\ref{l1s3}, Lemma~\ref{l2s3} and the proof of Theorem~\ref{NCPoiss} (i),

\begin{corollary} \la{c1s3}
For any $n$-cyclic (finite dimensional) graded algebra $A$, $\mathrm{HC}^{\bullet}(A)[2-n]$ has the structure of a graded Lie algebra. Moreover, for any finite dimensional $k$-vector space $V$,
$$(\mathrm{Tr}_V)_{\bullet}\,:\,\bSym(\mathrm{HC}^{\bullet}(A)) \rar \mathrm{H}_{\bullet}(\Omega(A^{\ast}),V)^{\GL} $$
is a morphism of graded $(2-n)$-Poisson algebras.
\end{corollary}

For example, when $n=2$ and $A$ is $2$-cyclic, then $\Omega(A^{\ast})$ is a double Poisson algebra by Theorem~\ref{t1s3}. Lemma~\ref{l1s3} implies that $\Omega(A^{\ast})$ acquires a noncommutative Poisson structure from its double Poisson structure. Lemma~\ref{l2s3} implies that $\mathrm{HC}^{\bullet}(A)$ has a graded Lie algebra structure and Theorem~\ref{NCPoiss} (i) implies that
$$(\mathrm{Tr}_V)_{\bullet}\,:\,\bSym(\mathrm{HC}^{\bullet}(A)) \rar \mathrm{H}_{\bullet}(\Omega(A^{\ast}),V)^{\GL} $$
is a morphism of graded Poisson algebras. There is no shortage of cyclic $2$-algebras: the cohomology $\mathrm{H}^{\bullet}(M,\c)$ of any compact smooth $2$-manifold $M$ is such an algebra.

\subsection{Derived Poisson structures on $k[x,y]$}

The construction in the previous subsection is interesting when $C:= k.a \oplus k.b \oplus k.s$ with $|a|=|b|=1$ and $|s|=2$ with $\Delta(a)=\Delta(b)=0$ and $\Delta(s)=a \otimes b -b \otimes a$. In this case, there is a natural isomorphism
$$
\Omega(C) \stackrel{\cong}{\rightarrow} R\ ,\quad s \mapsto t,\ a \mapsto x ,\  b \mapsto y \ ,
$$
where $ R:=k\langle x,y,t\rangle\,,|x|=|y|=0,\,|t|=1$ with differential given by $dt:=[x,y]$. Note that $R$ is a almost free resolution of $A:=k[x,y]$ in
$\DGA_k$.

One can check that there is exactly one cyclic structure of degree $-2$ on $C$ (up to multiplication by scalars)
with $\omega(a,a)=\omega(b,b)=\omega(\mbox{--},s)=0$ and $\omega(a,b)=1$. Similarly, there is exactly one cyclic structure of degree $-3$ on $C$ (up to multiplication by scalars) with $\tilde{\omega}(a,s)=\tilde{\omega}(b,s)=1$. By Theorem~\ref{t1s3},

\begin{lemma}
 The cyclic structure $\omega$ (resp., $\tilde{\omega}$) on $C$ induces a NC Poisson (resp., NC $(-1)$-Poisson) structure on $R$.
\end{lemma}

Since $R$ is an almost free resolution of
$\, A := k[x,y]\,$ (see~\cite{BKR}, Example 4.1), taking homology yields a graded Lie bracket
of degree $0$ on $\,\rHC_{\bullet}(A)\,$ induced by the cyclic structure $\omega$ on $C$:
\begin{equation}
\la{derpois}
\{\,\mbox{--}\,,\,\mbox{--}\,\}_{\natural}:\ \rHC_{\bullet}(A) \times \rHC_{\bullet}(A) \to \rHC_{\bullet}(A)\ ,
\end{equation}
which is thus an example of a derived Poisson structure on $A$.

This structure has a natural geometric interpretation. If we restrict \eqref{derpois} to
$ \rHC_0(A) = \bar{A} $, we get the usual Poisson bracket on polynomials associated to the
symplectic form $ dx \wedge dy $. The Lie algebra
$ (\bar{A}, \{\,\mbox{--}\,,\,\mbox{--}\,\}_{\natural}) $ is thus isomorphic to the Lie algebra of (polynomial)
symplectic vector fields on $ k^2 $. Now, if we identify $\, \rHC_1(A) = \Omega^1(A)/dA \,$ as in~\cite{BKR}, Example 4.1, then, for any $ \bar{f} \in \bar{A} $ and $ \bar{\alpha} \in \Omega^1(A)/dA $,
$$
\{\bar{f},\,\bar{\alpha}\}_{\natural}= {\mathscr L}_{\theta_f}(\alpha)\ ,
$$
where $ {\mathscr L}_{\theta_f} $ is the Lie derivative on $1$-forms taken along the
Hamiltonian vector field $ \theta_f $.
For example, if $f=x^p$ and $\alpha = y^qdx$, then $\bar{\alpha}$ corresponds to the class of the $1$-cycle $qy^{q-1}t$ in $R_{\n}$ (see~\cite{BKR}, Example 4.1). Note that
$$\{x^p, y^{q-1}t\}_{\n}\,=\, \sum_{i=1}^{q-1}px^{p-1}y^{q-1-i}ty^{i-1} \,\text{.}$$
Again, by~\cite{BKR}, Example 4.1, the image of the R.H.S of the above equation is identified with the class of the $1$-form $x^{p-1}y^{q-1}dx$ in $\rHC_1(A)$. Thus,
$$\{ x^p, y^qdx\}_{\n} \,=\, pqx^{p-1}y^{q-1}dx \,\text{.}$$
On the other hand, $\theta_f=px^{p-1}\frac{\partial}{\partial y}$, and hence,
$$
{\mathscr L}_{\theta_f}(\alpha)\,=\, pqx^{p-1}y^{q-1}dx \,\text{.}
$$
In addition, the restriction
of \eqref{derpois} to $ \rHC_1(A) $ is zero (for degree reasons). Thus, the graded Lie algebra
$\,\rHC_{\bullet}(A)\,$ is isomorphic to the semidirect
product $\,\bar{A} \ltimes (\Omega^1 A/dA) \,$, where $ \bar{A} $ is equipped with the
standard Poisson bracket and $\, \Omega^1(A)/dA \,$ is a Lie module over $ \bar{A} $ with
action induced by the Lie derivative on $ \Omega^1(A) $.
The Lie bracket \eqref{derpois} extends to the graded symmetric algebra $ \bSym[\rHC(A)] $ making it a Poisson algebra. Theorem~\ref{NCPoiss} implies that $\H_{\bullet}(A, V)^{\GL}$ has a (unique) graded Poisson structure such that
the trace map $(\Tr_V)_{\bullet}:\, \bSym[\rHC(A)] \to  \H_{\bullet}(A, V)^{\GL} $ is a morphism of Poisson algebras.

Similarly, taking homology yields a (graded) Lie bracket $\{\mbox{--},\mbox{--}\}_{\natural,\tilde{\omega}}$ on $\rHC_{\bullet}(A)[-1]$ induced by the cyclic structure $\tilde{\omega}$ on $C$.
 Of course, for degree reasons the restriction of $\{\mbox{--},\mbox{--}\}_{\natural,\tilde{\omega}}$ to $\rHC_0(A)$ is trivial. However, for $f \in \overline{A}$ and $\alpha \in \rHC_1(A) \cong \Omega^1(A)/dA$,
the geometric interpretation of $\{f,\alpha\}_{\n,\tilde{\omega}}$ remains mysterious.

We expect that the DG resolutions of algebras that are $n$-Calabi-Yau in the sense of Ginzburg
\cite{G1} (see also \cite{Ke}) have analogous noncommutative $(2-n)$-Poisson structures. In particular,
Ginzburg $2$-Calabi-Yau algebras are expected to have derived NC Poisson structures.

\subsection{Remarks on string topology}
Let $M$ be a smooth compact oriented manifold. Denote by $LM$ the free loop space of $M$.
In \cite{ChasSullivan}, M.~Chas and D.~Sullivan have shown that the $S^1$-equivariant homology
$ H_*^{S^1}(LM) $ of $ LM $ has a natural Lie algebra structure. Their construction uses (in an essential way)
the transversal intersection product of two chains in a manifold. Since the intersection product is only defined for transversal chains, it is difficult to realize the Lie algebra $ H_*^{S^1}(LM) $ algebraically. This is the subject
of string topology, which has become a very active area of research in recent years.

By a well-known theorem of K.-T. Chen \cite{Chen} and J. D. S. Jones \cite{J}, if $M$ is simply connected,
there is a quasi-isomorphism of complexes
$$
\mathrm{CC}_*(A(M))\stackrel{\simeq}\longrightarrow C^*_{S^1}(LM)\ ,
$$
where $ A(M)$ is any DG algebra model (de Rham, singular, PL forms etc.) for the
cochain complex of $M$. Similarly, using the methods of \cite{Chen} and \cite{J}, one can
construct a quasi-isomorphism
$$
C_*^{S^1}(LM)\stackrel{\simeq}\longrightarrow \mathrm{CC}_*(C(M))\ ,
$$
where $C(M)$ is any DG coalgebra model for the chain complex of $M$.
On the other hand, Lambrechts and Stanley \cite{LambStan} have recently shown that
for $M$ simply connected, there is a finite-dimensional DG coalgebra $ C(M)$
with a cyclically invariant nondegenerate pairing, that is quasi-isomorphic to the singular
chain complex of $M$. Further, the nondegenerate pairing on $ C(M)$ gives the
intersection product pairing at the homology level.
Combining these results with our Corollary~\ref{c1s3}, we thus obtain a Lie algebra structure
on the cyclic homology of $C(M)$ that realizes the Lie algebra of Chas and Sullivan
({\it cf.} \cite{CEG}).

Besides Sullivan and his school, Blumberg, Cohen and Teleman
are carrying out a project that aims to systematically lift the interesting structures on $LM$ to the
path space $PM$ of $M$ (see  \cite{BCT}). More precisely, they associate to $M$ a category
where the objects are the points of $M$ and the space of morphisms between two
objects is a (for example, singular) chain complex of the space of paths connecting them.
Our present paper has essentially the same starting point as \cite{BCT}: we have shown that the Lie algebra of
string topology on $H_*^{S^1}(LM)$ arises from the NC Poisson structure of the path space on
$M$. More precisely, the cyclic homology of $\Omega(C)$, which is exactly the cyclic homology
of the above described category, is isomorphic to $\mathrm{HC}_*(C)$ (see Corollary~\ref{c2s3}), and hence is isomorphic to $ H_*^{S^1}(LM) $. This clarifies the relation between the above mentioned theorem of Jones and a well-known theorem of Goodwillie (see \cite{Good}).

\end{document}